\documentclass[a4paper,10pt]{amsart}
\usepackage{bbm,amssymb,amsmath,pinlabel}

  \usepackage{hyperref}
 
\input xy
\xyoption{all}

\usepackage{color}

\newtheorem{theorem}{Theorem}[section]
\newtheorem{lemma}[theorem]{Lemma}
\newtheorem{proposition}[theorem]{Proposition}
\newtheorem{corollary}[theorem]{Corollary}

\theoremstyle{definition}
\newtheorem{definition}[theorem]{Definition}
\newtheorem{remark}[theorem]{Remark}
\newtheorem{example}[theorem]{Example}

\numberwithin{equation}{section}

\newcommand{\gras}[1]{{\mathbb #1}} 
\newcommand{\C}{\gras{C}} 
 
\newcommand{\Q}{\gras{Q}} 
\newcommand{\N}{\gras{N}} 
\newcommand{\R}{\gras{R}}

\newcommand{\cS}{\mathcal{S}}

\newcommand{\calE}{\mathcal{E}}
\newcommand{\calA}{\mathcal{A}}
\newcommand{\calB}{\mathcal{B}}
\newcommand{\calP}{\mathcal{P}}

\def\elem(#1,#2){  \{ \frac{#1}  {\overline {\ #2\ }} \} }

\title[Fibonacci numbers and lattice structures for plane branches]
    {Fibonacci numbers and  self-dual lattice structures 
    for plane branches}
\author{Mar\'{\i}a Pe Pereira}
    \address{Instituto de Ciencias Matem\'aticas, ICMAT, Madrid, Spain y 
       Institut de Maths. de Jussieu, Paris, France.}
     \email{mariapepereira@hotmail.com}
\author{Patrick Popescu-Pampu}
   \address{Universit{\'e} Lille 1, UFR de Maths., B\^atiment M2\\
     Cit\'e Scientifique, 59655, Villeneuve d'Ascq Cedex, France.}
   \email{patrick.popescu@math.univ-lille1.fr}
\date{10th of March 2013}

\begin{document}

\begin{abstract}
    Consider a plane branch, that is, an irreducible germ of curve on a smooth complex analytic 
    surface. We define its {\em blow-up complexity} as the number of blow-ups 
    of points necessary to achieve its minimal embedded resolution. We show that there are 
    $F_{2n-4}$ topological types of blow-up complexity $n$, where $F_{n}$ is the 
    $n$-th Fibonacci number. We introduce complexity-preserving 
    operations on topological types which increase the multiplicity and we deduce that the 
    maximal  multiplicity for a plane branch of blow-up complexity $n$ is $F_n$. It is achieved by 
    exactly two topological types, one of them being distinguished as 
    the only type which maximizes the Milnor number. 
    We show moreover that there exists a natural partial order relation on the set  
    of topological types of plane branches of blow-up complexity $n$, making this set 
    a {\em distributive lattice}, that is, any two of its elements admit an infimum and a 
    supremum, each one of these operations beeing distributive relative to the second one. 
    We prove that this lattice admits a {\em unique} order-inverting bijection. As this 
    bijection is involutive, it defines a {\em duality} for topological types of plane branches.   
    The type which maximizes the Milnor number is also the maximal element 
    of this lattice and its dual is the unique type with minimal Milnor number.  
    There are $F_{n-2}$ self-dual topological types of blow-up complexity $n$. 
    Our proofs are done by encoding the topological types by 
    the associated {\em Enriques diagrams}. 
\end{abstract}

\maketitle

\section{Introduction}  
\label{sec:Intro}

Let $C$ be a plane branch, that is, an irreducible 
germ of an analytic curve on a smooth analytic surface $\mathcal{S}$. 
It is a classical fact that one may get a canonical embedded resolution of it  
by successively blowing up the singular points of the strict transform of $C$. 
We say that the number of 
blow-ups needed to arrive at the minimal embedded resolution 
is the {\em blow-up complexity} of $C$. 
This notion is not to be confused with that of {\em resolution complexity} introduced 
by L\^e and Oka in \cite{LO 95}. 

The blow-up complexity is a topological invariant of the pair $(\mathcal{S},C)$. 
It is then natural to 
compare it with more common invariants, as its {\em multiplicity} and its  
{\em Milnor number}. 

We were surprized to discover that, 
{\em if one fixes a blow-up complexity $n$, 
then the maximal multiplicity is equal to the $n$-th Fibonacci number $F_n$}, and that 
{\em there 
are exactly two topological types realizing this maximum} (recall that the 
{\em Fibonacci sequence} $(F_n)_{n \geq 0}$ is defined by the initial conditions 
$F_0 =0, \: F_1 =1$ and the recursive relation $F_{n+1} = F_n + F_{n -1}$, for all $n \geq 1$). 
One may discriminate these two multiplicity-maximizing types using the Milnor number: {\em one 
of them is the unique topological type with maximal Milnor number among plane branches of 
blow-up complexity $n$} (see Theorem \ref{theo:mult}). 

This motivated us to study in more detail the set $\calE_n$ of embedded topological types of 
plane branches with blow-up complexity $n$. We discovered a second appearance of 
the Fibonacci numbers: {\em the cardinal of $\calE_n$ is equal to $F_{2n - 4}$} 
(see Theorem \ref{cardtypes}). But 
$\calE_n$ should not be thought only as a set: we found out a natural partial order 
on $\calE_n$  which makes it a {\em distributive lattice}, that is, any two elements have an infimum 
and a supremum, each one of these operations being distributive with respect to the other one 
(see Proposition \ref{distrlat}). 
This partial order has an absolute maximum, which is the  
topological type of maximal Milnor number alluded to before. 
There is also an absolute minimum, which may be 
characterized as the unique topological type of blow-up complexity $n$ with multiplicity $2$ 
(it is the simple singularity $\mathbb{A}_{2n -4}$). 

This symmetry between the minimum and the maximum extends 
to a {\em duality} on embedded topological types of plane branches: for each $n$, there is 
an order-inverting involution on $\calE_n$ (see Definition \ref{def:dual}). 
This involution is the {\em unique} bijection 
of $\calE_n$ on itself which inverts the partial order (see Theorem \ref{uniqdual}). 
The Fibonacci numbers appear for a third time: {\em there are $F_{n-2}$ self-dual topological 
types of complexity $n$} (see Proposition \ref{selfdual}).

As far as we know, no such lattice structures or duality were known before. 
See Remark \ref{rem:notsame} for some comments about the classical projective 
duality of plane curves.

\medskip
Let us explain now our way to work with embedded topological types. 
There are various classical ways to encode them; 
the most common ones are the sequence of characteristic Newton-Puiseux 
exponents and the weighted dual graph of the minimal embedded resolution. 
Nevertheless, here we describe structures on $\calE_n$ which we discovered and we 
believe are most clearly understandable using what was, historically speaking, 
the first graphical encoding of the topological type of a plane branch: its {\em Enriques diagram}. 

An Enriques diagram associated to a branch 
is a decorated graph homeomorphic to an interval, whose vertices are 
labeled by the infinitely near points appearing during the canonical process of 
embedded resolution by point blow-ups. 
By Enriques' convention, the edges  are either {\em curved} or {\em straight}, 
and moreover one tells if at the junction point of two successive straight edges 
the diagram is broken or not: we speak then about {\em breaking} versus {\em neutral} vertices 
(see Definition \ref{Enriques}).  All our results are proved by 
studying carefully those decorations. For instance, the duality expresses itself most 
easily as a symmetry between curved edges and breaking vertices: note that both ends 
of a curved edge are neutral and both edges adjacent to a breaking vertex are straight. 

The reader who wants to get an overview of the structure of the paper may read the short 
introductory paragraphs of the various sections.  The final Remark \ref{rem:rinrem} 
explains how we were led to discover our results.  As we intend this paper to be understandable 
to both singularists and combinatorialists, we wrote a rather detailed section 
with basic material about infinitely near points and Enriques diagrams (Section 
\ref{sec:Enriques}), which is standard in singularity theory, and another detailed 
section with basic material about partial order relations and lattices (Section \ref{sec:Basic}), 
which is standard in combinatorics.

\medskip

We conclude this introduction with a few words of explanation about our use of the term  
``complexity''. Following Matveev's convention in \cite{M 90} (see also \cite{M 07}), 
an invariant of a class of objects 
may be considered as a {\em complexity measure} if the set of isomorphism classes of objects 
with a given invariant is {\em finite}. Our {\em blow-up complexity} satisfies this condition, 
as well as the {\em Milnor number}, if we look at the embedded topological types of plane branches as objects. But the {\em multiplicity} or the {\em resolution complexity} 
of L\^e and Oka do not.

\medskip
\section{The Enriques diagram and the multiplicity sequence}
\label{sec:Enriques}

In this section we recall basic vocabulary about infinitely near points, as well as the 
equivalent notions of {\em Enriques diagram} and {\em multiplicity sequence} 
associated to a plane branch. The reader who is not familiar with the process of 
{\em point blowing up} and of the way its iteration leads to resolutions of plane curve 
singularities, may gain a lot of intuition as well as technical skills by consulting 
Brieskorn and Kn{\"o}rrer's book \cite{BK 86}. 
\medskip

Let $(\cS,O)$ be a germ of smooth complex analytic surface. A {\bf model} over $(\cS,O)$ 
is a morphism $(\Sigma, E) \stackrel{\pi}{\longrightarrow} (\cS,O)$ obtained as a sequence of blowing-ups 
of points above $O$. 
Its {\bf exceptional divisor} is the reduced curve $E := \pi^{-1}(O)$. 
If  $(\Sigma_i, E_i) \stackrel{\pi_i}{\longrightarrow} (\cS,O)$ for $i =1,2$ are two models 
and $P_i \in E_i$, we say that 
the points $P_1, P_2$ are {\em equivalent} when the bimeromorphic map $\pi_2^{-1} \circ \pi_1$ 
is an isomorphism in a neighborhood of $P_1$. An {\bf infinitely near point} of $O$ is 
an equivalence class of points on various models over $(\cS,O)$. By abuse of language, we 
will say also that  any one of its representatives is an infinitely near point of $O$. 

Denote by $\mathcal{C}_O$ the set of all infinitely near points of $O$. If $P \in \mathcal{C}_O$, 
we denote by $E(P)$ the smooth irreducible rational curve obtained by blowing up $P$. Of 
course, this blow up procedure has to be done in a model, but the various exceptional curves 
obtained like this get canonically identified by the bimeromorphic maps $\pi_2^{-1} \circ \pi_1$. 
We make an abuse of notation and we denote also by 
$E(P)$ its strict transform in further blow-ups. 

If $P, Q \in \mathcal{C}_O$, we say that $Q$ is {\bf proximate} to $P$ if $Q \in E(P)$, 
and we write $Q \mapsto P$. As the exceptional divisor on any model has normal crossings, 
any one of its points lies on one or two of its irreducible components, that is, 
it is proximate either to one or to two other points of 
$\mathcal{C}_O$. In the first case it is called a {\bf free point} over $O$ and in the second one 
a {\bf satellite} over $O$. 
\medskip

Let $(C,O) \hookrightarrow (\cS,O)$ be a {\bf branch}, that is, a reduced irreducible germ of 
complex analytic curve. In the sequel, in order to insist on the fact that the surface $\cS$ is smooth, 
we will say that $(C, O)$ is a {\bf plane branch}. A model 
$(\Sigma, E)  \stackrel{\pi}{\longrightarrow} (\cS,O)$ is called an {\bf embedded resolution} of $(C,O)$ 
if the total transform $\pi^{-1}(C)$ is a normal crossings divisor.  There exists a unique 
{\em minimal} embedded resolution, obtained recursively by blowing up the unique point 
of the last defined model where the total transform  of $C$ has not a normal crossing. 

Denote by $(P_i)_{0 \leq i \leq n-1}$ the finite sequence of infinitely near points of $O$ which 
are blown up in order to achieve the minimal embedded resolution of $(C,O)$. Therefore 
$P_0=O$. Moreover, either $C$ is smooth, in which case $n=0$, or it is singular and $n \geq 3$. 
In this second case, $P_{n-1}$ is a satellite point over $O$ and the strict transform of 
$C$ passing through it is smooth and transversal to both components of the exceptional divisor.

\begin{definition} \label{blowcomp}
    Let $(C,O)$ be a plane branch. We say that the {\bf blow-up complexity} of $C$ 
    is the number of infinitely near points of $O$ which have to 
    be blown up in order to achieve its minimal embedded resolution. 
\end{definition}

With the previous notations, the blow-up complexity of $C$ is equal to $n$. 

 For all $i \in  \{0, \dotsc,  n-1 \}$, denote by $C_i$ the strict transform of $C$ passing through 
 $P_i$ 
 and by $m_i(C) \in \{1, 2, \dotsc, m_0(C)\}$ 
 the multiplicity of $C_i$ at $P_i$. Therefore, $m_0(C)$ denotes the multiplicity of $C$ at $O$. 
 As the multiplicity of a strict transform 
 drops (not necessarily strictly) when one does another blow-up, this sequence is decreasing. 
 As the strict transform of $C$ passing through 
 $P_{n-1}$ is smooth, one has $m_{n-1}(C)=1$.
 
 \begin{definition} \label{multseq}
     The decreasing sequence of positive integers $(m_0(C), \dotsc, m_{n-1}(C))$ is called 
    the  {\bf multiplicity sequence} of the plane branch $C$ with blow-up complexity $n$. 
 \end{definition}
 
 One has the following {\em proximity relations} between the terms of the multiplicity sequence 
 (see \cite[Prop. 3.5.3]{CA 00}):
 
 \begin{proposition} \label{proxrel}
     For each $i \in \{0, \dotsc, n-1\}$, 
     the multiplicity $m_i(C)$ is equal to the sum of multiplicities of the strict transforms of $C$ 
        which pass through points proximate to $P_i$. That is: 
        $$m_i(C) = \sum_{P_j \mapsto P_i} m_j(C) .$$
 \end{proposition}
 
 The Milnor number $\mu(C)$ of $(C,O)$ (introduced first in \cite{M 68}  for germs of isolated 
 complex algebraic hypersurface singularities of arbitrary dimension) 
 may be expressed in the following way in terms of the associated multiplicity  sequence  
 (see \cite[Proposition 6.4.1]{CA 00}  or \cite[Prop. 6.5.9]{W 04}):
 
 \begin{proposition} \label{compMiln}
    From the multiplicity sequence $(m_i(C))_{0 \leq i \leq n-1}$ associated to the plane branch  
    $(C,O)$, one can compute the Milnor number of $(C,O)$ by the formula: 
       $$ \mu(C)=\sum_{0 \leq i \leq n-1} m_i(C) \cdot (m_i(C) - 1).$$
  \end{proposition}
  
  Note that the previous formula shows that the Milnor number is necessarily {\em even}. 
  In fact, $\mu(C)/ 2$ is equal to the more classical $\delta$-invariant of the branch (see for instance 
  \cite[Prop. 6.3.2]{W 04}). In the sequel we will also need (see \cite[Example 6.5.1]{W 04}):
  
  \begin{proposition} \label{compMiln2}
    Let $C$ be a plane branch defined by the equation $x^a - y^b =0$, where 
    $a,b$ are coprime positive integers. Then $\mu(C) = (a-1)(b-1)$. 
  \end{proposition}
 
 Let us describe now the geometric object by which we will represent all along 
 the paper the topological type of $(\mathcal{S},C)$: its {\em Enriques diagram}, 
 introduced first in  \cite[Libro Quarto, Cap. 1, Sec. 8]{EC 18} (see also 
 Casas' book \cite[Section 3.9]{CA 00}). It is a finite graph homeomorphic to an  interval and 
 enriched with {\em decorations} of the edges and of the vertices: 
 the edges may be {\em curved/straight} and the vertices {\em neutral/breaking}:

 \begin{definition} \label{Enriques}
    If $v, w$ are two vertices of a graph homeomorphic to an interval, 
    we denote by $[vw]$ the segment joining them. 
   The {\bf Enriques diagram} $\epsilon(C)$ of the plane branch $(C,O)$ with 
   blow-up complexity $n$ is a 
   graph homeomorphic to an interval, 
   whose vertices $(v_i)_{0 \leq i \leq n-1}$ correspond bijectively to the 
   infinitely near points $(P_i)_{0 \leq i \leq n-1}$ and whose edges 
   $(e_i)_{1 \leq i \leq n-1}$ are indexed such that $e_i = [v_{i-1} v_i]$.   
   The edges are either {\bf curved} or {\bf straight}  
   and successive straight edges make either a straight path or a broken one, 
   according to the following rules:
    \begin{enumerate}
       \item The edge $e_i$ is curved 
              if and only if $P_i$ is a free point.
       \item Assume that $P_i$ has at least two proximate points in the sequence 
          $(P_i)_{0 \leq i \leq n-1}$. If $P_{i + 1}, P_{i+2}, \dotsc, P_k$ are all the 
          points proximate to $P_i$ (therefore $k\geq i+2$), 
          then the path $[v_{i+1} v_k]$ is straight. Moreover, $[v_{i+1} v_k]$ 
         is a {\bf maximal straight path}, that is, adding any one of its adjacent edges 
         $e_{i+1}= [v_i v_{i + 1}]$ and 
          $e_{k+1}= [v_k v_{k + 1}]$, one does not get a straight path. 
          If $[v_iv_{i+1}]$ is also straight, 
          we say that $[v_{i}v_k]$ is a 
          path which is broken at the {\bf breaking vertex} $v_{i+1}$.   The vertices 
         which are not breaking ones are called {\bf neutral}.      
     \end{enumerate}
   We say that $\epsilon(C)$ is an {\bf Enriques diagram of complexity} $n$ and 
   we denote by $\calE_n$ the set of isomorphism classes of such diagrams. 
 \end{definition}

  \begin{remark} \label{rem:Key}
    Note that for $n \geq 3$, the edge $e_1$ is always curved and the edge 
    $e_{n-1}$ is always straight, as explained before Definition \ref{blowcomp}.
     If a vertex is breaking, then the two adjacent edges 
    are straight. In a dual manner, if an edge is curved, then its 
    vertices are neutral. In particular, only the 
    vertices $\{v_2, \dotsc , v_{n-2} \}$ and the edges 
    $\{e_2, \dotsc , e_{n-2} \}$ may have both decorations 
    when we vary the diagram among the elements of $\calE_n$. 
  \end{remark}
    
    In the previous definition, the attributes ``curved'' and ``straight'' associated to the edges and 
    ``breaking'' or ``neutral'' associated to the vertices are 
    purely combinatorial. Nevertheless, their concrete meaning gives a way to 
    represent an Enriques diagram as a piecewise smooth embedded arc in the plane. 
    
    \begin{example} \label{ex:Enr}
         In Figure \ref{fig:Seqblow} are represented schematically the exceptional divisors of 
         a sequence of point blowing-ups, as well as the associated Enriques diagram. 
         Its complexity is $6$. In order not to charge the drawing, we did not represent 
         the strict transforms of a branch having this embedded resolution process. 
    \end{example}

\begin{figure}[h!] 
\vspace*{6mm}
\labellist \small\hair 2pt 
\pinlabel{$O = P_0$} at 4 344
\pinlabel{$P_1$} at 110 376
\pinlabel{$E(P_0)$} at 115 436
\pinlabel{$P_2$} at 275 356
\pinlabel{$E(P_1)$} at 190 390
\pinlabel{$E(P_0)$} at 240 436
\pinlabel{$P_3$} at 455 385
\pinlabel{$P_4$} at 422 160
\pinlabel{$P_5$} at 123 155
\pinlabel{$E(P_0)$} at 400 436
\pinlabel{$E(P_2)$} at 460 436
\pinlabel{$E(P_1)$} at 350 390
\pinlabel{$E(P_0)$} at 355 255
\pinlabel{$E(P_1)$} at 310 210
\pinlabel{$E(P_2)$} at 475 200
\pinlabel{$E(P_3)$} at 423 114
\pinlabel{$E(P_0)$} at 75 255
\pinlabel{$E(P_1)$} at 10 220
\pinlabel{$E(P_4)$} at 177 138
\pinlabel{$E(P_2)$} at 175 200
\pinlabel{$E(P_3)$} at 133 114

\pinlabel{$v_0$} at 140 -10
\pinlabel{$v_1$} at 195 77
\pinlabel{$v_2$} at 219 -10
\pinlabel{$v_3$} at 282 -10
\pinlabel{$v_4$} at 360 -10
\pinlabel{$v_5$} at 420 -10

\pinlabel{$e_1$} at 132 52
\pinlabel{$e_2$} at 237 52
\pinlabel{$e_3$} at 266 16
\pinlabel{$e_4$} at 320 43
\pinlabel{$e_5$} at 376 20

\endlabellist 
\centering 
\includegraphics[scale=0.55]{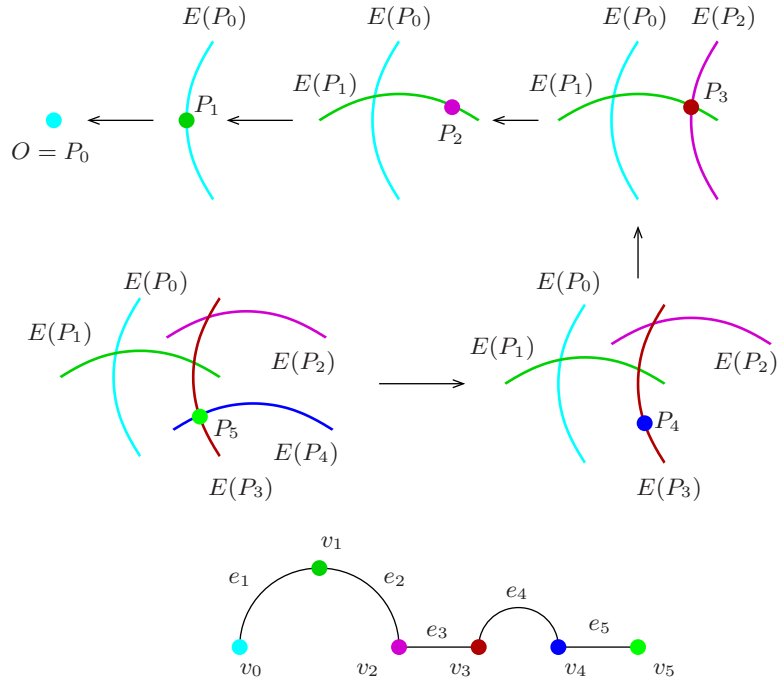} 
\caption{A sequence of blow-ups and its Enriques diagram} 
\label{fig:Seqblow} 
\end{figure}

  \begin{remark}   \label{differdef}
    In the original definition by Enriques, as a rule for plane representation 
    one also supposed that any curved edge  
    formed a smooth arc with the previous edge, be it either curved of 
    straight. Moreover, Enriques chose to draw perpendicularly the two maximal straight 
    segments adjacent to a breaking vertex. Here we do not keep  
    those supplementary conventions, as they do not give more information. Moreover, 
    in this way we gain more flexibility for our drawings. 
  \end{remark}

The following proposition is an immediate consequence of Proposition \ref{proxrel} 
and shows that the Enriques diagram contains the same information as the multiplicity 
sequence $(m_i(C))_{0 \leq i \leq n-1}$. 
Nevertheless, in the sequel it will be important to think about both of them simultaneously: 
the multiplicity sequence will be seen as a {\bf multiplicity function} 
$\underline{m}$ defined on the set of vertices of the Enriques diagram. 
To simplify notations, we will write simply $m_i$ 
instead of $\underline{m}(v_i)$.  We say that $m_0$ is the {\bf initial multiplicity} of 
an Enriques diagram.  Proposition \ref{compMiln} shows that the Milnor number of a plane branch 
is determined by the associated Enriques diagram. Therefore, we will speak also about the 
{\bf Milnor number} $\mu$ of such a diagram. 

\begin{proposition} \label{compmult}
   Consider an Enriques diagram of complexity $n\geq 3$. 
   Then $m_{n-1}=1$ and 
   for each $i \in \{0, \dotsc, n - 1\}$, the multiplicity $m_i$ may be computed in the following way 
   from the multiplicities $(m_j)_{j > i}$: 
     \begin{enumerate}
         \item  If there is no maximal straight path of the form $[v_{i+1} v_j]$, with $j > i + 1$, 
            then:
                 \begin{equation}\label{mult1}m_i = m_{i+1}.\end{equation}
   
         \item \label{second} If $[v_{i+1} v_j]$, with $j > i + 1$, 
         is a maximal straight path of the Enriques diagram, 
         then:
             \begin{equation}\label{multgen}
                    m_i = \sum_{k = i + 1} ^j m_k.
              \end{equation}            
         More precisely:
             \begin{itemize}   
                \item {\bf if $e_{j+1}= [v_j v_{j+1}]$ is curved}, then $m_k=m_j$ 
                  for all $k=i+1,\dotsc,j$ and:  
            \begin{equation}\label{mult2} m_i = (j-i)\cdot  m_{i+1},\end{equation}
               \item {\bf if $e_{j+1}$ is straight}, 
                then $m_k=m_{i+1}$ for all $k=i+1,\dotsc,j-1$ and: 
            \begin{equation}\label{mult3} m_i=m_j+(j-i-1) \cdot m_{i+1}.\end{equation}
               \end{itemize}
           \end{enumerate}
\end{proposition}

As explained  in \cite[Sections 3.5, 3.6, 5.5]{W 04}, 
the following invariants associated to 
a complex plane branch contain the same information:
   \begin{itemize} 
        \item its multiplicity sequence;
        \item its sequence of generic Newton-Puiseux exponents;
        \item the weighted dual graph of the exceptional divisor of its minimal embedded resolution;
        \item its embedded topological type (that is, the topology of the associated knot in the 
           $3$-sphere). 
  \end{itemize}
  One may also consult \cite{BK 86} for the relation between the last three view-points 
and \cite{PP 11} for the relation between the third one and the Enriques diagram.
  
 Therefore, all such objects parametrize the embedded topological types of plane branches. Nevertheless, 
 as the notion of topological type is specific to $\C$, while the other ones are applicable to 
 curves defined over arbitrary algebraically closed fields with characteristic zero, 
 we will speak about the set 
 of {\em combinatorial types} instead of {\em embedded topological types} of plane branches. 

In general, when different logically equivalent encodings of some class of objects are available, 
they are 
not equivalent from the viewpoint of adaptability to specific situations. For instance, 
in this paper we will show that Enriques diagrams are especially adapted to emphasize a 
hidden lattice structure on the set of combinatorial types of plane branches with fixed 
blow-up complexity.

\medskip

\section{The number of combinatorial types of branches with fixed blow-up complexity}
\label{sec:Card}

We begin the study of the sets $\calE_n$ of combinatorial types of branches of 
fixed blow-up complexity. 
In this section we show that their cardinals are Fibonacci numbers. 
\medskip

\begin{theorem} \label{cardtypes}
  The number of combinatorial types of plane branches with 
  blow-up complexity $n\geq 3$ is 
   the Fibonacci number $F_{2n-4}$.
\end{theorem}

\begin{proof}
We decompose each set 
$\calE_n$ into the disjoint union of two subsets $\calA_n$ and $\calB_n$, and we compute by 
induction the pair of cardinals $(  | \calA_n| ,  | \calB_n| )$, where:

\begin{itemize}
   \item  $\calA_n$ is the set of Enriques diagrams of complexity $n$ whose vertex 
     $v_{n-2}$ is neutral. 

   \item $\calB_n$ is the set of Enriques diagrams of complexity $n$ whose vertex 
     $v_{n-2}$ is breaking. 
\end{itemize}

Each diagram of $\calE_{n+1}$ can be obtained from a diagram of $\calE_n$, 
which is determined by the decorations of the vertices 
$v_2, \dotsc, v_{n-2}$ and of the edges $e_2, \dotsc, e_{n-2}$, 
by adding the information about the decorations of $v_{n-1}$ and $e_{n-1}$. 
We count the number of elements of $\calE_{n+1}$ by looking at  the possible 
ways to complete a given diagram of $\calE_n$ (recall Remark \ref{rem:Key}). 
This number changes from $\calA_n$ to $\calB_n$:

\begin{itemize}
    \item Each graph of $\calA_n$ can be completed in 3 ways as an Enriques diagram of 
       complexity $n+1$: {\em either} with a neutral vertex $v_{n-1}$, and with an edge $e_{n-1}$ 
       {\em either} curved {\em or} straight, {\em or} with a breaking vertex $v_{n-1}$ and a 
        straight edge. 

    \item Each graph of $\calB_n$ can be completed in 2 ways, {\em either} with a breaking vertex 
       {\em or} with a neutral vertex $v_{n-1}$, but always with a straight edge $e_{n-1}$, 
       such that $v_{n-2}$ keeps being a breaking vertex.
\end{itemize}

We deduce that: 
    $$
     \left\{
     \begin{array}{l}
      | \calA_{n+1} | = 2  | \calA_n| +  | \calB_n| , \\
      | \calB_{n+1}| = | \calA_n | + | \calB_n|,
     \end{array} \right.
         \mbox{ for all } n \geq 3.
   $$

As $| \calA_3 |  = 1$ and $| \calB_3 |  = 0$, which is seen immediately by inspection, the 
previous recursive relations allow to prove immediately by induction on $n\geq 3$ that:
   $$
     \left\{
     \begin{array}{l}
      | \calA_n | = F_{2n -5} , \\
      | \calB_n | = F_{2n-6},
     \end{array} \right.
         \mbox{ for all } n \geq 3.
   $$
    
 Therefore  $| \mathcal{E}_n | =  | \calA_n | + | \calB_n| = F_{2n -4}$. 
\end{proof}

In Theorem \ref{theo:mult} we will see a second appearance of Fibonacci numbers 
related to the sets $\calE_n$.

\medskip
\section{Multiplicity increasing operators}
\label{sec:Mult}

In this section we introduce two types of partially defined operators on Enriques diagrams and 
we prove that they make strictly increase the multiplicity function and the Milnor number. 
Moreover, we describe the cases when also the initial multiplicity increases strictly.  
\medskip

\begin{definition} \label{def:operators}
    Let $n \geq 3$. 
   We introduce the following partially defined operators on $\calE_n$ 
   (see Figures \ref{fig:Strdef} and \ref{fig:Brdef}, in which are 
indicated the different possibilities for the adjacent edges):
 \begin{enumerate}
   \item[(i)] Suppose that the diagram $\epsilon \in \mathcal{E}_n$ 
   is such that its edge $e_p$ is curved for some $p \geq 2$. 
   Let  $s_p(\epsilon)$ be the 
   diagram obtained from $\epsilon$ by declaring the edge $e_p$ straight, 
   alined with $e_{p-1}$ if this is straight in $\epsilon$ and declaring the vertex $v_p$ 
   breaking if $e_{p+1}$ is straight. We say that $s_p$ is the {\bf straightening operator} 
   at the edge $e_p$. 
   
   \item[(ii)] Suppose that the diagram $\epsilon \in \mathcal{E}_n$ 
   is such that the path $[v_{p-1} v_{p+1}]$ is straight for some $p \in \{2, \dotsc , n-2\}$. 
   Let $b_p(\epsilon)$ be the diagram obtained from $\epsilon$ by declaring $v_p$ 
   a breaking vertex. We say that $b_p$ is the {\bf breaking operator} at the vertex $v_p$. 
  \end{enumerate}
  
  In both cases, all the decorations of the vertices and the edges which are not 
  mentioned are left unchanged. 
\end{definition}

One has therefore the straightening operators $s_2, \dotsc , s_{n-2}$ and the breaking ones 
  $b_2, \dotsc , b_{n-2}$. 
They are {\em partially defined} in the sense that $s_p(\epsilon)$ 
  is defined only if  $e_p$ is a curved edge of the diagram $\epsilon$ and 
  $b_p(\epsilon)$ is defined only if $v_p$ is a neutral vertex between two straight edges.

\begin{figure}[h!] 
\vspace*{6mm}
\labellist \small\hair 2pt 
\pinlabel{$s_p$} at 183 114

\pinlabel{$v_{p-2}$} at 30 144
\pinlabel{$v_{p-2}$} at 27 64
\pinlabel{$v_{p-1}$} at 100 68
\pinlabel{$v_p$} at 224 68
\pinlabel{$v_{p + 1}$} at 315 65
\pinlabel{$v_{p + 1}$} at 267 34
\pinlabel{notice this} at 123 7
\pinlabel{disymmetry} at 123 -10

\endlabellist 
\centering 
\includegraphics[scale=0.60]{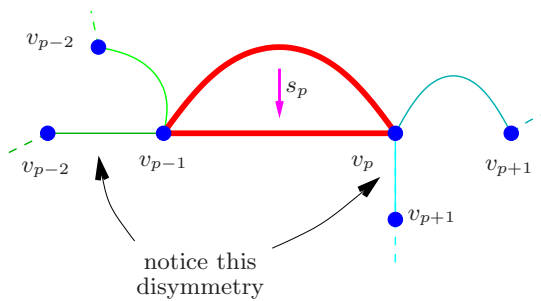}
\vspace*{5mm} 
\caption{The straightening operator $s_p$} 
\label{fig:Strdef} 
\end{figure}

\begin{figure}[h!] 
\vspace*{6mm}
\labellist \small\hair 2pt 
\pinlabel{$b_p$} at 263 24

\pinlabel{$v_{p-1}$} at 28 -10
\pinlabel{$v_{p-1}$} at 380 -10
\pinlabel{$v_p$} at 90 -10
\pinlabel{$v_p$} at 447 -10
\pinlabel{$v_{p + 1}$} at 146 -10
\pinlabel{$v_{p + 1}$} at 470 60

\endlabellist 
\centering 
\includegraphics[scale=0.60]{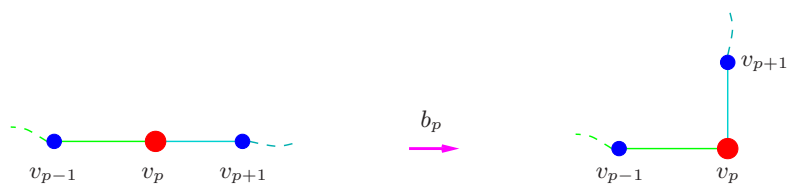}
\vspace*{5mm} 
\caption{The breaking operator $b_p$} 
\label{fig:Brdef} 
\end{figure}

  The next theorem states that the multiplicities and Milnor numbers increase when one applies 
  either type of operator. Moreover, we describe the situations when the increase is not strict.

\begin{theorem} \label{theo:incr}
   Let $\epsilon \in \calE_n$ be an Enriques diagram with multiplicity function 
     $\underline{m}$. Denote by $\epsilon' \in \calE_n$ a diagram obtained 
   from $\epsilon$ by applying either a straightening or a breaking operator. 
   Denote by $\underline{m}'$ its multiplicity function. Then:
   \begin{enumerate}
     \item   \label{strictfunct}
         $\underline{m}' > \underline{m}$ (that is, $m_j' \geq m_j$ for all $j \in 
         \{0, \dotsc , m_{n-1}\} $ and there is at least one strict inequality).  
     
     \item  \label{caseq}
   $m_0' \geq m_0$. 
     The two multiplicities are equal if and only if $\epsilon'$ 
     is obtained from $\epsilon$ by applying 
     a breaking operator $b_p$ and if moreover $e_{p-1}$ 
   is a curved edge of $\epsilon$. 
   
     \item \label{strictMiln}
     The Milnor number of the diagram $\epsilon'$ is strictly greater than the 
        Milnor number of the starting diagram $\epsilon$. 
    \end{enumerate}
\end{theorem}

\begin{proof}  The proof of (\ref{strictfunct}) and (\ref{caseq}) follows from 
 the repeated use of the proximity relations stated in Proposition \ref{proxrel}, under the more 
explicit forms of Proposition \ref{compmult}: the relations (\ref{mult1})-(\ref{mult3}). Let us 
develop this. 

Assume that $\epsilon' = s_p(\epsilon)$ or $\epsilon' = b_p(\epsilon)$ for some 
$p \in \{2, \dotsc , n-2 \}$. This implies, by Proposition \ref{proxrel}, that 
$m_i' = m_i$ for all $i \in \{p,  \dotsc , n-1\}$.

Let $l > 0$ be maximal such that $[v_{p -l} v_{p-1}]$ is a straight path of the 
diagram $\epsilon$. Notice that, by Proposition \ref{compmult}, the rules of 
computation of $m_i$ and $m_i'$ starting from the knowledge of the multiplicities 
at the vertices of higher subindex are the same if $i < p -l-1$. Therefore, 
in order to prove that $\underline{m}' > \underline{m}$, it will be enough to 
check the inequalities $m_i' \geq m_i$ for all $i \in \{p -l-1, \dotsc , p -1 \}$. 
We will use the following immediate consequence of Proposition \ref{proxrel}:
   \begin{equation}  \label{eq:lemprox}
        \mbox{ if } m_i' >  m_i \mbox{ for some } i \leq p - l -1, \mbox{ then } 
            m_j ' > m_j \mbox { for all } j \in \{ 0 ,  \dotsc , i \}.
   \end{equation}

 We treat now separately  the two types of operators.

\medskip 
{\bf $*$ The case of a straightening operator}: assume that $\epsilon'= s_p(\epsilon)$  
 (see  Figure \ref{fig:Straightopbis}).  
 \medskip

\begin{figure}[h!] 
\vspace*{6mm}
\labellist \small\hair 2pt 
\pinlabel{$s_p$} at 292 114
\pinlabel{$e_p$ in $\epsilon$} at 290 155
\pinlabel{$e_p$ in $\epsilon'$} at 280 64
\pinlabel{these edges may be} at 184 -10
\pinlabel{either curved or straight} at 184 -30

\pinlabel{$v_{p-l - 1}$} at -20 33
\pinlabel{$v_{p-l}$} at 10 107
\pinlabel{$v_{p-2}$} at 130 67
\pinlabel{$v_{p-1}$} at 200 70
\pinlabel{$v_{p + 1}$} at 380 34
\pinlabel{$v_p$} at 370 87

\endlabellist 
\centering 
\includegraphics[scale=0.60]{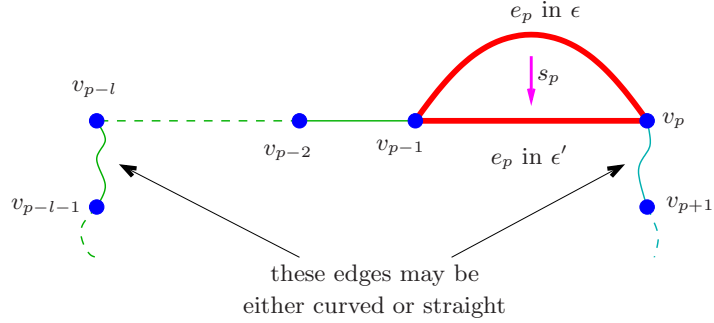}
\vspace*{5mm} 
\caption{The diagrams $\epsilon$ and $\epsilon'= s_p(\epsilon)$} 
\label{fig:Straightopbis} 
\end{figure}

 By (\ref{multgen}) and the fact that $m_i' = m_i$ for all $i \geq p$, we get 
 $m_{p-1}' = m_{p-1}$. We deduce from (\ref{mult1}) 
        that $m_j' = m_{p-1}' = m_{p-1} =  m_j$ for all $ j \in \{ p - l, \dotsc, p - 1 \}$. Then, by 
        (\ref{mult3}):
     $$m_{p - l -1}' = m_p' + l \cdot m_{p -l}' > l \cdot m_{p -l}' = l \cdot m_{p -l} = m_{p - l -1}.$$
          By (\ref{eq:lemprox}), we see that 
                 $m_j' > m_j$ for all $j \in \{ 0, \dotsc , p -l -1 \}.$
               This ends the proof for this operator.

\medskip 
 {\bf $*$ The case of a breaking operator}: assume that $\epsilon'= b_p(\epsilon)$ 
    (see Figure \ref{fig:Breakopbis}). 
    \medskip

\begin{figure}[h!] 
\vspace*{10mm}
\labellist \small\hair 2pt 
\pinlabel{$b_p$} at 110 150
\pinlabel{$v_p$} at 170 310
\pinlabel{$v_p$} at 260 33
\pinlabel{$v_{p-1}$} at 114 310
\pinlabel{$v_{p-1}$} at 180 70
\pinlabel{$v_{p + 1}$} at 222 310
\pinlabel{$v_{p + 1}$} at 265 102
\pinlabel{$v_{p -l}$} at 6 310
\pinlabel{$v_{p-l}$} at 73 70
\pinlabel{$v_{p -l-1}$} at -20 250
\pinlabel{$v_{p -l -1}$} at 45 15
\pinlabel{$v_{p + h}$} at 312 310
\pinlabel{$v_{p + h}$} at 260 175
\pinlabel{$v_{p + h + 1}$} at 350 250
\pinlabel{$v_{p + h + 1}$} at 310 180

\pinlabel{these edges may be} at 158 246
\pinlabel{either curved or straight} at 158 226

\endlabellist 
\centering 
\includegraphics[scale=0.60]{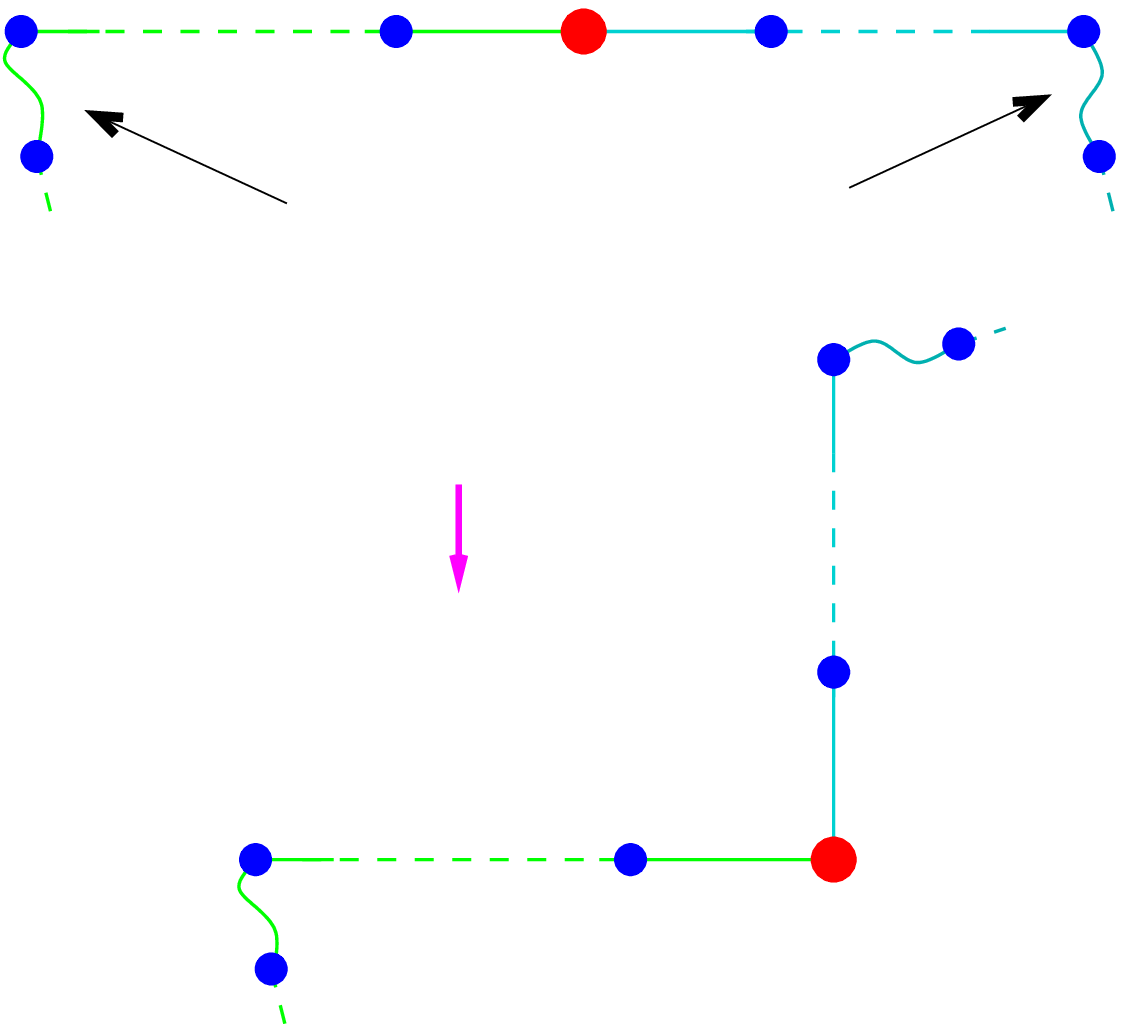}
\vspace*{1mm} 
\caption{The diagrams $\epsilon$ and $\epsilon'= b_p(\epsilon)$} 
\label{fig:Breakopbis} 
\end{figure} 

 Let $h > 0$ be maximal such that $[v_p v_{p + h}]$ is a straight interval of the 
diagram $\epsilon$ (and therefore also of the diagram $\epsilon'$). 
        
        By (\ref{mult3}), $m_{p-1}' = m_{p + h}' + h \cdot m_p' > m_p' = 
            m_p = m_{p-1}$. From (\ref{mult1}) we deduce that $m_j' = m_{p-1}' > m_{p-1} = m_j$ 
            for all $j \in \{ p-l, \dotsc, p -1 \}$. 
            
            As a consequence: $m_{p -l -1}' = m_p' + l \cdot m_{p -l}' = m_p' + l \cdot m_{p-1}' = 
                m_p' + l\cdot ( m_{p+h}' + h \cdot m_p') = 
                (1 + lh)\cdot m_p' + l \cdot m_{p + h}' =   (1 + lh)\cdot m_p + l \cdot m_{p + h} 
                \geq (l + h) \cdot m_p + m_{p + h} = m_{p -l -1}$. 
               
              Therefore, $m_{p -l -1}'  \geq m_{p -l -1}$, with an equality precisely when 
             $1 + lh = l + h$ and $l =1$ hold simultaneously. Therefore, 
              $m_{p -l -1}'  = m_{p -l -1}$ if and only if $l=1$. 
             
             Let us consider now two subcases. 
             
             \medskip
             
             \begin{itemize}
                \item {\bf Assume that $l > 1$.} Then $m_{p -l -1}'  > m_{p -l -1}$, and 
                    (\ref{eq:lemprox}) implies that 
                 $m_j' > m_j$ for all $j \in \{0, \dotsc, p - l -1 \}$. 
                
                \item {\bf Assume that $l =1$.} We consider again two subcases:
                
                     \begin{itemize}
                          \item {\bf Assume that $e_{ p -1}$ is curved, that is, that $v_{p-1}$ 
                               is neutral.} 
                             Therefore 
                              $m_{p-2}'= m_{p-2}$. By Proposition \ref{compmult}, 
                              we see that for any $j \in \{0, \dotsc, p - 2 \}$, 
                              the formulae expressing $m_j$ and $m_j'$ in terms 
                              of the multiplicities $m_{i > j}$ and $m_{i > j}'$ are the same 
                              and involve only subindices 
                              $i \leq p-2$.  This implies, by descending induction on $j$, 
             that $m_j' = m_j$ for all $j \in \{0, \dotsc, p - 2 \}$.

                          \item {\bf Assume that $e_{ p -1}$ is straight,
                              that is, that $v_{p-1}$ is breaking}  (see Figure \ref{fig:Lastcase}). 
                              Let  $k \geq 2$ be maximal such that 
                             $[v_{p-k} v_{p - 1}]$ is a straight interval of the diagram $\epsilon$.

\begin{figure}[h!] 
\vspace*{6mm}
\labellist \small\hair 2pt 
\pinlabel{$s_p$} at 287 124
\pinlabel{this edge may be} at 260 0
\pinlabel{either curved or straight} at 260 -20

\pinlabel{$v_{p-k - 1}$} at 24 94
\pinlabel{$v_{p-k - 1}$} at 370 94
\pinlabel{$v_{p-k}$} at 97 73
\pinlabel{$v_{p-k}$} at 450 73
\pinlabel{$v_{p-1}$} at 41 158
\pinlabel{$v_{p-1}$} at 390 158
\pinlabel{$v_{p + 1}$} at 190 138
\pinlabel{$v_{p + 1}$} at 465 207
\pinlabel{$v_p$} at 135 138
\pinlabel{$v_p$} at 492 138

\endlabellist 
\centering 
\includegraphics[scale=0.60]{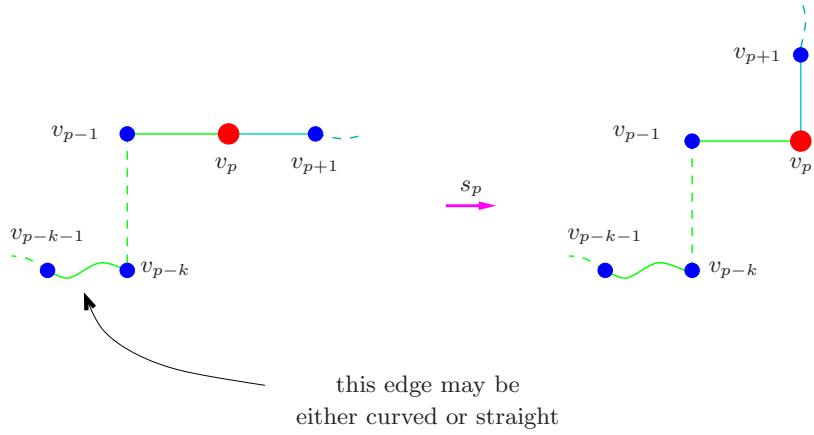}
\vspace*{5mm} 
\caption{The breaking operation when $v_{p-1}$ is a breaking vertex} 
\label{fig:Lastcase} 
\end{figure}

            By (\ref{mult1}), we get that 
            $m_j' = m_j$ for all $j \in \{p-k, \dotsc, p - 2 \}$. 
            But then 
              $m'_{p -  k -1} = m_{p -1}' + (k-1) \cdot m_{p -k}' > m_{p-1} + (k-1) \cdot m_{p-k}$, 
              which by (\ref{eq:lemprox}) implies that 
              $m_j' > m_j \mbox{ for all } j \in \{0, \dotsc, p-k-1 \}.$

                     \end{itemize}
              \end{itemize}

       This finishes the proof of the points (\ref{strictfunct}) and (\ref{caseq}) of the theorem. 
    Point (\ref{strictMiln}) is then a direct consequence of them and of Proposition \ref{compMiln}. 
  \end{proof}

As a consequence of the previous proof, one may describe the set of vertices at which the multiplicity 
function increases {\em strictly}:

\begin{proposition} \label{precision}
   Let $\epsilon, \epsilon' \in \calE_n$ be such that either $\epsilon' =s_p(\epsilon)$ or 
    $\epsilon' =b_p(\epsilon)$. Let 
   $l \geq 1$ be maximal such that $[v_{p-l} v_p]$ is straight and, 
   in the case where $l=1$ and $e_{p-1}$ is straight, $k \geq 2$ is maximal 
   such that $[v_{p-k} v_{p-1}]$ is straight.  Denote also by 
   $J\subset \{0, \dotsc, n-1\} $ the set of indices $i$ such that 
   $m_i' > m_i$. Then:
     \begin{itemize}
         \item  $J = \{0, \dotsc, p-l-1 \}$ if $\epsilon' = s_p(\epsilon)$.
                  
         \item $J = \{0, \dotsc, p-1 \}$ if $\epsilon' = b_p(\epsilon)$ and $l >1$.
         
         \item $J = \{p-1 \}$ if $\epsilon' = b_p(\epsilon)$, $l=1$ and $e_{p-1}$ is curved.
         
         \item $J = \{0, \dotsc, p-k-1 \} \cup \{ p-1 \}$ if $\epsilon' = b_p(\epsilon)$, $l=1$ and 
              $e_{p-1}$ is straight.
     \end{itemize}
\end{proposition}

This shows in particular that, in the case when the initial multiplicity does not change, the 
multiplicity function changes at only one vertex.

\medskip
\section{Extremal multiplicities and Milnor numbers for fixed complexity}
\label{sec:Extr}

In this section we prove that for a fixed complexity $n \geq 3$, 
the minimal multiplicity is $2$ and the maximal one is the $n$-th Fibonacci number $F_n$. 
The number $F_n$ appears as the natural candidate for the biggest value 
$m_0$ attained in a sequence $(m_0, \dotsc , m_{n-1} =1)$ generated in reverse 
order by using the rules described in Proposition \ref{compmult}. 
Anyway, the complete proof requires some meticulosity and we do it here as a 
consequence of Theorem \ref{theo:incr}. In particular, we show that the maximal 
multiplicity $F_n$ is achieved by exactly two combinatorial types in $\calE_n$ 
and that one of them is distinguished by the property 
of maximizing also the Milnor number. Meanwhile, the minimal multiplicity $2$ is 
obviously achieved by only one combinatorial type. 
\medskip

\begin{definition} \label{def:Extrdiag}
For all $n \geq 3$, denote by $\alpha_n, \omega_n, \pi_n \in \calE_n$ the 
diagrams represented in Figure  \ref{fig:Extrem}.  That is:
  \begin{itemize}
      \item $\alpha_n$ has all its edges $e_2, \dotsc, e_{n-2}$ curved (therefore  
          all its vertices $v_2, \dotsc, v_{n - 2}$ are neutral). It is the Enriques diagram 
          of the plane branch defined by the equation $x^{2n-3} -y^2 =0$, that is, of the simple 
          singularity $\mathbb{A}_{2n - 4}$.
      \item $\omega_n$ has all its vertices $v_2, \dotsc, v_{n - 2}$ breaking  
            (consequently, all the edges $e_2, \dotsc, e_{n-2}$ are straight). It is the Enriques diagram 
            of the  plane branch defined by the equation $x^{F_{n+1}}-y^{F_n}=0$. 
      \item $\pi_n$ is identical to $\omega_n$,  excepted that $v_2$ is a neutral 
              vertex. It is the Enriques diagram of the plane branch 
              defined by the equation $x^{F_{n-2}+F_n}-y^{F_n}=0$. 
  \end{itemize}
 \end{definition}

The notations $\alpha_n, \omega_n, \pi_n$ are explained in Remark \ref{rem:alphaomega}. 
In order to get the defining equations, one may use the transformation rules described 
for instance in \cite[Section 3]{W 04}. The previous Enriques diagrams may be characterized 
in the following way:

\begin{theorem}\label{theo:mult}
  The diagrams $\alpha_n, \omega_n, \pi_n$ satisfy the following extremal properties 
  among Enriques diagrams of blow-up complexity $n$: 
  \begin{enumerate}
      \item $\alpha_n$ is the unique diagram with minimal multiplicity, which is equal to $2$, 
          and with minimal Milnor number, equal to $2n-4$. 

      \item $\omega_n, \pi_n$ are the only diagrams with maximal multiplicity, 
         equal to the $n$-th Fibonacci number $F_n$. 
      
      \item $\omega_n$ is the unique diagram with maximal Milnor number, equal to 
          $(F_{n+1} -1)(F_n -1)$. 
   \end{enumerate} 
\end{theorem}

\begin{figure}[h] 
\vspace*{6mm}
\labellist \small\hair 2pt 
\pinlabel{$v_0$} at -10 4
\pinlabel{$v_1$} at -7 30
\pinlabel{$v_{n-2}$} at 120 4
\pinlabel{$v_{n-1}$} at 196 4

\pinlabel{$v_0$} at 266 -10
\pinlabel{$v_1$} at 285 -10
\pinlabel{$v_2$} at 304 -10
\pinlabel{$v_{n-2}$} at 345 92
\pinlabel{$v_{n-1}$} at 382 90

\pinlabel{$v_0$} at 433 -10 
\pinlabel{$v_1$} at 453 -10
\pinlabel{$v_2$} at 472 -10
\pinlabel{$v_{n-2}$} at 530 92
\pinlabel{$v_{n-1}$} at 570 92

\pinlabel{{\Large $\alpha_n$}} at 79 -30
\pinlabel{{\Large $\omega_n$}} at 323 -30
\pinlabel{{\Large $\pi_n$}} at 480 -30

\endlabellist 
\centering 
\includegraphics[scale=0.65]{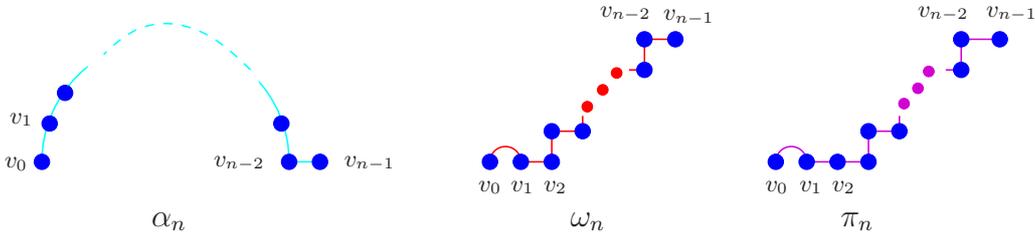}
\vspace*{5mm} 
\caption{The diagrams $\alpha_n$, $\omega_n$ and $\pi_n$} 
\label{fig:Extrem} 
\end{figure}

 \begin{proof} All the statements of this theorem are rapid consequences of 
    Theorem \ref{theo:incr}. 
    
    \medskip
  {\bf -- Proof of (1):} The diagram $\alpha_n$ has multiplicity $2$ and 
      Milnor number $2n-4$, as may be seen from the defining equation and 
      Proposition \ref{compMiln2}. Any 
      other diagram $\epsilon \in \calE_n$ may be obtained from 
      $\alpha_n$ by a sequence of straightening and breaking operators, with at 
      least one straightening operator being applied. 
      By Theorem \ref{theo:incr}, we deduce that $m_0(\epsilon) > m_0(\alpha_n)$ 
      and $\mu(\epsilon) > \mu(\alpha_n)$ for all $\epsilon \in \calE_n$.

      \medskip
   {\bf -- Proof of (2):} We see that $\omega_n = b_2(\pi_n)$ and that $e_1$ is curved 
   in $\pi_n$. By Theorem \ref{theo:incr}, we deduce that $\omega_n$ and $\pi_n$ 
   have the same multipicity and that $\mu(\omega_n) > \mu(\pi_n)$, 
   as may be seen also from the defining equations. 
   
   Assume now that $\epsilon \in \calE_n$ is a diagram different from them. Therefore, 
   one may obtain $\omega_n$ from it by applying a sequence of at least two 
   straightening or breaking operators. As $\epsilon \neq \pi_n$, 
   either one of them is a 
   straightening operator, or there is a breaking one $b_i$ among them 
   such that $e_{i-1}$ is not curved. Therefore, Theorem \ref{theo:incr} implies that 
   $m_0(\epsilon) < m_0(\omega_n)$. 
   
       \medskip
    {\bf -- Proof of (3):} The reasoning is the same, but even simpler, as applying an 
    operator makes the Milnor number increase {\em strictly}. The given expression of 
    $\mu(\omega_n)$ may be deduced from the defining equation and Proposition 
    \ref{compMiln2}. 
\end{proof}

As an immediate consequence of the previous theorem we get:

\begin{corollary}
    Any plane branch of multiplicity $m \geq 2$ has blow-up complexity at least $n$, 
    where the integer $n \geq 3$ is such that $F_{n-1}< m \leq F_n$. 
 \end{corollary}

\begin{remark} \label{fixMilnor}
    Trying to find the extremal number of characteristic Newton-Puiseux exponents for 
    a fixed blow-up complexity or the extremal blow-up complexity for a fixed Milnor number, 
    one gets less interesting results. We leave the following as exercises for the reader:
    
    $\bullet$ Among the combinatorial types of plane branches with blow-up complexity 
         $ n \geq 3$,  the maximal 
        number of Newton-Puiseux exponents is the integral part $[\frac{n-1}{2}]$. 
         This maximum is achieved 
        once for $n$ {\em odd} and $n-2$ times for $n$ {\em even}. 
    
    $\bullet$ 
    The maximal blow-up complexity of plane branches with Milnor number $\mu \in 2 \N^*$ 
    is equal to $2 + \frac{\mu}{2}$. It is achieved by exactly one combinatorial type 
    (and one analytical type), that of the simple plane branch $\mathbb{A}_{\mu}$, 
    which has one Newton-Puiseux 
    exponent equal to $\frac{\mu + 1}{2}$. 
\end{remark}

\medskip
\section{The self-dual lattice structure on $\calE_n$}
\label{sec:Latt}

In this section we give another way of codifying an Enriques diagram (therefore, the combinatorial 
type of a plane branch) 
as a subset of the set of symbols of vertices and edges. This makes appear 
a natural partial order relation on $\calE_n$, coming from the inclusion relation 
among such subsets.  We show that this relation is 
a {\em lattice structure}.  
Moreover, we show that there exists an order-inverting involution 
on $\calE_n$. This allows to speak about the {\em dual} of any combinatorial type 
of plane branch. We present then a third occurrence of the Fibonacci numbers in our context: 
they appear as the cardinals of the sets of self-dual combinatorial types for each fixed 
complexity. The few notions about lattices which are used here are explained 
in the Section \ref{sec:Basic}. 
\medskip

Let us fix $n \geq 3$. Denote by:
    $$ S_n : = \{v_2,\dotsc,v_{n-2}, e_2,\dotsc,e_{n-2}\}$$
the set of {\em symbols} of the vertices and edges of the Enriques diagrams of complexity $n$ 
whose decoration is undetermined (recall Remark \ref{rem:Key}). That is, we think 
about the {\em name} of each edge or vertex, not about the geometric object itself. 
We encode now each Enriques 
diagram by {\em the subset of $S_n$ consisting of the symbols of its straight edges and of its 
breaking vertices}:

\begin{definition}  \label{code}
   The {\bf code} $\chi(\epsilon) \subseteq S_n$ of an Enriques diagram $\epsilon 
     \in \calE_n$, or of the corresponding combinatorial type of plane branch,  
      is the set of symbols of its straight edges and of its breaking vertices.  
      We denote by $\mathcal{K}_n \subseteq \calP(S_n)$ the set of such codes.
\end{definition} 

Here $\calP(S_n)$ denotes the {\em power set} of $S_n$, that is, the set of its subsets. 

\begin{example} \label{ex:first}
    For instance, if $\epsilon$ is the Enriques diagram of Example \ref{ex:Enr}, 
    which is of blow-up complexity $6$, then its code is $\{ e_3\}$. 
    For any complexity $n \geq 3$, one has $\chi(\alpha_n) = \emptyset, \ 
    \chi(\omega_n) = S_n$ and $ \chi(\omega_n) = S_n \setminus \{v_2 \}.$
\end{example}

As may be easily shown using the Definition \ref{Enriques}, the codes of 
the Enriques diagrams may be characterized in the following way:

\begin{lemma}  \label{lem:Image}
   A subset $\chi$ of $S_n$ is the code of an Enriques diagram $\epsilon \in \calE_n$ 
   if and only if it has one of the  following equivalent properties:
     \begin{itemize}
\item[(a)] for all $i=2,\dotsc,n-2$, if $v_i\in \chi$ then $e_i,e_{i+1}\in \chi$; 
\item[(b)] for all $j=2,\dotsc,n-2$, if $e_j\notin \chi$ then $v_j, v_{j+1}\notin \chi$. 
\end{itemize}
\end{lemma}

On the power set $\calP(S_n)$, let us consider the inclusion $\subseteq$ as partial order 
relation. Endowed with it, $(\calP(S_n), \subseteq)$ is a lattice (and even a Boolean algebra). 
Restrict this partial order to the set of codes of the Enriques diagrams of complexity $n$. 

\begin{lemma} \label{lem:Stable}
     The subset $\mathcal{K}_n$ of $\calP(S_n)$ is stable under the intersection 
     $\cap$ and union $\cup$ operations. That is, it is a sublattice 
     of $(\calP(S_n), \subseteq)$. 
\end{lemma}

\begin{proof}
  It is immediate to check that both the intersection and the union of two subsets 
     of $S_n$ which satisfy either one of the conditions (a) and (b) of Lemma 
     \ref{lem:Image} satisfy again that condition. 
\end{proof}

\begin{remark}
    The subset $\mathcal{K}_n \subseteq \calP(S_n)$ is not stable by the operation 
    of taking the 
    complement, therefore it is not a sub-Boolean algebra of 
    $(\calP(S_n), \subseteq)$. For instance, 
    $S_4 = \{v_2, e_2 \}$ and $\{e_2\} \in \mathcal{K}_4$ but 
    $S_4 \setminus \{e_2\} = \{v_2\} \notin \mathcal{K}_4$, by Lemma \ref{lem:Image}. 
\end{remark}

We are ready to define the lattice structure on the set $\calE_n$ of Enriques 
diagrams of complexity $n$:

\begin{definition}  \label{def:Latstr}
   The {\bf staircase partial order relation}  
   $\preceq$ on the set $\calE_n$ of Enriques diagrams of complexity $n$ 
   is defined by saying that, for any two diagrams $\epsilon, \epsilon' \in 
   \calE_n$, one has 
   $ \epsilon \preceq \epsilon'$  if and only if the code of  $\epsilon' $
      contains the code of $ \epsilon. $
\end{definition}

Our motivation for choosing this name comes from 
the fact that an Enriques diagram $\epsilon_2$ 
 is greater than another one $\epsilon_1$ of the same complexity 
 if and only if  $\epsilon_2$ has all the straight edges 
 and breaking vertices of  $\epsilon_1$, and maybe more.  
 That is, if and only if  $\epsilon_2$ looks more like the staircase diagram $\omega_n$ 
 than $\epsilon_1$. 

By the Lemma \ref{lem:Stable}, the staircase relation endows 
$\calE_n$ with a lattice structure. As this lattice is finite, it has an  
absolute minimum and an absolute maximum. 
Their characterization is a first immediate consequence of the previous 
definition and of the Definition \ref{def:Extrdiag} of the diagrams $\alpha_n$ and $\omega_n$:

\begin{proposition}  \label{prop:minmax}
    The minimum of $(\calE_n, \preceq)$ is the diagram $\alpha_n$ with 
    $\chi(\alpha_n) = \emptyset$ and the maximum is the diagram 
    $\omega_n$ with $\chi(\omega_n)= S_n$. More generally, 
    if $\epsilon'$ is obtained from $\epsilon$ by applying a straightening 
    or a breaking operator, then $\epsilon \prec \epsilon'$. 
\end{proposition}

\begin{proof}
    The first statement results from Example \ref{ex:first}. The second one results from the fact 
     that $\chi(s_p(\epsilon))$ is equal either to $\chi(\epsilon) \sqcup \{e_p\}$ or to 
      $\chi(\epsilon) \sqcup \{ v_p, e_p\}$ and that 
      $\chi(b_p(\epsilon)) = \chi(\epsilon) \sqcup \{v_p\}$ (here $\sqcup$ denotes a disjoint 
      union).
\end{proof}

\begin{remark} \label{rem:alphaomega}
   This proposition motivated us to choose the notations $\alpha_n$ and 
    $\omega_n$, as $\alpha$ is the first letter in the Greek alphabet and 
    $\omega$ is the last one. Concerning the third diagram $\pi_n$ appearing 
    in the statement of Theorem \ref{theo:mult}, we chose its name as 
    $\pi$ is the Greek analog of the initial of ``predecessor'': indeed, 
    $\pi_n$ is a predecessor of $\omega_n$ for the relation $\preceq$. 
\end{remark}

Notice that there is a symmetry between the two equivalent characterizations of 
$\mathcal{K}_n$ stated in Lemma \ref{lem:Image}:  one switches them by interchanging  
the dimensions (edges $\leftrightarrow$ vertices), the types (straight $\leftrightarrow$ 
neutral, curved $\leftrightarrow$ breaking), as well as the order of indexing. More 
precisely:

\begin{definition}  \label{def:dual}
     Let $\epsilon$ be an Enriques diagram of complexity $n \geq 3$. The 
     {\bf dual Enriques diagram} $\Delta_n(\epsilon)$ is defined by:
  $$\begin{array}{l}
           \bullet \ v_k \mbox{ is a breaking vertex in } \Delta_n(\epsilon) 
                 \mbox{ if and only if } e_{n-k} 
                   \mbox{ is a curved edge in } \epsilon; \\
           \bullet \ e_k \mbox{ is a straight edge in } \Delta_n(\epsilon) 
                 \mbox{ if and only if } v_{n-k} 
                   \mbox{ is a neutral vertex in } \epsilon.
      \end{array}$$
\end{definition}

\begin{example} \label{ex:dualEnr}
  Let us consider the Enriques diagram $\epsilon$ 
   of Example \ref{ex:Enr}. As $\chi(\epsilon)= \{ e_3\}$, 
  we get $\chi(\Delta_6(\epsilon)) = \{e_2, e_3, e_4, v_2, v_4\}$. The associated Enriques 
  diagram $\Delta_6(\epsilon)$ is drawn in Figure \ref{fig:Dualdiag}.
\end{example}

\begin{figure}[h!] 
\vspace*{6mm}
\labellist \small\hair 2pt 
\pinlabel{$v_0$} at 5 -10
\pinlabel{$v_1$} at 59 -10
\pinlabel{$v_2$} at 114 -10
\pinlabel{$v_3$} at 85 61
\pinlabel{$v_4$} at 91 114
\pinlabel{$v_5$} at 167 90
\pinlabel{$e_1$} at 31 46

\endlabellist 
\centering 
\includegraphics[scale=0.60]{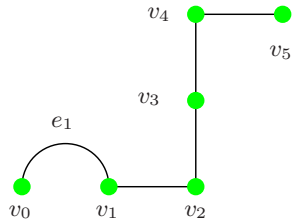} 
\caption{The Enriques diagram dual to the diagram of Figure \ref{fig:Seqblow}} 
\label{fig:Dualdiag}
\end{figure}

One may describe the duality map $\Delta_n : \calE_n \longrightarrow \calE_n$ 
 more geometrically as follows:
  \begin{itemize}
      \item Think about $\epsilon \in \calE_n$ as a cell decomposition of the 
          underlying segment, which is oriented by the chosen order of its vertices.
          
      \item Take the dual cell decomposition $\Delta_n(\epsilon)$ from the topological 
      viewpoint, endowed with the opposite orientation. 
      
      \item Look only at the cells of $\Delta_n(\epsilon)$ which are dual to those of 
         $\epsilon$ which may have both decorations. Decorate them by respecting the 
         following associations of the decorations of dual cells: curved/straight 
            $\leftrightarrow$ breaking/neutral.
  \end{itemize}

\begin{example} \label{ex:dualcells} 
An example of complexity $12$ 
is drawn in Figure \ref{fig:Dualcell}. The vertical lines connect cells which are 
dual to each other. 
The arrows indicate the orientations of the two cell decompositions 
associated to the numberings of their vertices. 
\end{example}

\begin{figure}[h!] 
\vspace*{6mm}
\labellist \small\hair 2pt 
\pinlabel{$v_0$} at 4 220
\pinlabel{$v_1$} at 40 220
\pinlabel{$v_{10}$} at 370 164
\pinlabel{$v_{11}$} at 410 184
\pinlabel{$e_1$} at 20 265

\pinlabel{$v_0$} at 420 54
\pinlabel{$v_1$} at 383 54
\pinlabel{$v_{10}$} at 50 145
\pinlabel{$v_{11}$} at 12 125
\pinlabel{$e_1$} at 400 105

\endlabellist 
\centering 
\includegraphics[scale=0.60]{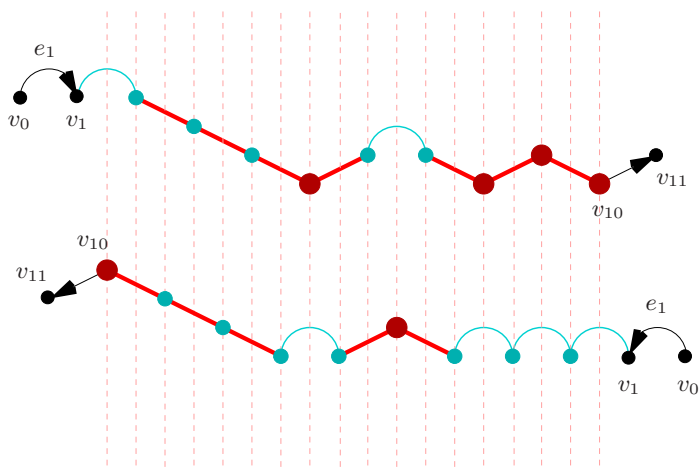} 
\caption{Dual Enriques diagrams as dual cell complexes} 
\label{fig:Dualcell}
\end{figure}

It is an 
immediate  consequence of Definition \ref{def:dual} that:

\begin{corollary} \label{cor:reverse}
   The duality map $\Delta_n : \calE_n \to \calE_n$ reverses the partial order 
    $\preceq$. 
\end{corollary}

In fact, $\Delta_n$ {\em is characterized by the previous property}: {\em it is the unique 
bijection of the set $\calE_n$ on itself which reverses the staircase partial order}. 
We will present the proof of this fact in Section \ref{sec:Uniq} 
(see Theorem \ref{uniqdual}), after having recalled 
some generalities about partial order relations, lattices and their Hasse diagrams in 
Section \ref{sec:Basic}. 

\begin{remark}  \label{rem:notsame}
    As proved by Wall \cite{W 94} (see also \cite[Section 7.4]{W 04}), if one starts 
    from a germ $(C,O)$ of an irreducible complex analytic curve in the complex 
    projective plane $\mathbb{P}^2$, then the germ of its projective dual 
    $(\check{C}, t_O)$ at the tangent line $t_O$ to $C$ at $O$ is strongly similar to 
    $C$. More precisely, their strict transforms by the blow-ups of $O$ and 
    $t_O$ respectively have the same combinatorial type. This is enough to see that 
    our notion of duality does not correspond to this projective duality. Indeed, 
    if the Enriques diagram $\epsilon(C)$ of $(C,O)$ is that of Example 
    \ref{ex:Enr}, that of its strict transform after one blow-up is obtained by removing
     $e_1$ (and renumbering consequently the remaining vertices and edges). 
     As shown by Example \ref{ex:dualEnr}, this new diagram  is not isomorphic to the 
     one obtained by the analogous procedure from $\Delta_6(\epsilon(C))$ (here one has 
     also to make $e_1$ curved). 
     But a more fundamental difference between the two notions of duality is that 
     the combinatorial type of the projective dual is not determined by the 
     combinatorial type of the initial branch (see \cite[Example 7.4.1]{W 04}).
\end{remark}

Let us present a third appearance of Fibonacci 
numbers in our context (recall from the proof of Theorem \ref{cardtypes} that we 
denote by $\mathcal{A}_{n}$  the set of  Enriques diagrams in $\mathcal{E}_{m+2}$ 
such that $v_{n-2}$ is neutral):

\begin{proposition} \label{selfdual}
   The set $\mathcal{D}_n$ of self-dual elements of $\mathcal{E}_{n}$ is in a natural bijection with 
   $\mathcal{E}_{(n+2)/2}$ if $n$ is even and with $\mathcal{A}_{(n+3)/2}$ if $n$ is odd. 
   In particular,  there are $F_{n-2}$ self-dual combinatorial types of plane branches 
   of blow-up complexity $n$.
\end{proposition}

\begin{proof} Informally speaking, the idea is that  ``the first half'' of a self-dual Enriques diagram 
   determines its second half. Moreover, one knows how both halves are joined. This allows 
   to get a bijection between the set of self-dual diagrams of given complexity and a subset 
   of the diagrams of approximately half the complexity. In order to make this argument precise, 
   we describe it according to the parity of $n$. 

       Consider first a self-dual diagram $\epsilon \in \mathcal{D}_{2m}$. If 
       its edge $e_m$ were curved, 
        then its adjacent vertex $v_m$ would be neutral. Therefore, in the dual 
        diagram the edge $e_m$ would be straight, which would 
        contradict the self-duality. This shows that  
    $e_{m}\in \chi(\epsilon)$, which implies by the same argument that  
    $v_{m}\notin \chi(\epsilon)$. As a consequence, the map:
    $$\begin{array}{ccc}
           \mathcal{D}_{2m} & \longrightarrow & \calE_{m+1}  \\
             \epsilon  &   \longrightarrow &  \mbox{the diagram whose  code is } 
                 \chi(\epsilon)\cap S_{m+1} 
          \end{array}$$  
    is bijective.

    We can argue similarly for the self-dual diagrams of $\mathcal{E}_{2m+1}$. 
    Given $\epsilon \in \mathcal{D}_{2m+1}$, one sees that its vertex $v_{m}$ 
    is neutral.  
Analogously to the previous case,  we get a bijection: 
   $$\begin{array}{ccc}
           \mathcal{D}_{2m+1} & \longrightarrow & \calA_{m+2}  \\
             \epsilon  &   \longrightarrow &   \mbox{the diagram whose  code is }  
                 \chi(\epsilon)\cap S_{m+2} 
          \end{array}.$$

Using Theorem  \ref{cardtypes} as well as the computation of the cardinality of 
$\mathcal{A}_n$ done in its proof, we get now the cardinality 
of the set of self-dual diagrams of given complexity, as stated in the proposition.
\end{proof}

\medskip

\section{Basic facts about posets and  lattices}
\label{sec:Basic}

In this section we explain basic facts about \emph{partially ordered sets}, 
{\em lattices} and {\em Hasse diagrams}. We will apply these notions 
only to {\em finite sets}. For a more detailed introduction to lattices, one may consult 
Birkhoff and Bartee's book \cite{BB 70}. 
For much more details about lattices and the historical development of their theory, 
one may consult Birkhoff \cite{B 48} and Gr\"{a}tzer \cite{G 71}. 
\medskip

\begin{definition} \label{posetdef}
   A {\bf partial order} $\preceq$ on a set $S$ is a binary relation 
   which is reflexive, antisymmetric and transitive. 
   A {\bf partially ordered set (poset)} is a set endowed with a partial order. 
  A {\bf hereditary subset} (or {\bf ideal}) of a poset is such that each time it  
   contains some element, it also contains all the elements which 
   are less or equal to it. 
\end{definition}

As is customary for the usual partial order $\leq$ on $\R$, if 
$\preceq$ is a partial order on a set $S$, 
we denote by $\prec$ the (transitive) binary relation defined by: 
  $$ a \prec b \ \Leftrightarrow \ (a \preceq b \mbox{ and } a \neq b). $$
Formally speaking, this second relation is not a partial order, as it is 
not reflexive. Nevertheless, common usage allows to speak also about 
{\em the partial order $\prec$}. Notice that the usual notation of inclusion 
of subsets of a given set does not respect this convention: 
$A \subset B$ does not imply that $A \neq B$. 

\begin{definition}
  Let $(E, \preceq)$ be a poset. The associated {\bf successor relation}, 
  denoted $\prec_s$, is the binary relation defined by the condition 
 that, for any $a,b \in E$, one has $a \ \prec_s \ b$ if and only if 
 $ a \prec b$ and if there is no element $c \in E$ such that $a \prec c \prec b$. 
  We say then that $b$ {\bf is a successor of} $a$ or that $a$ {\bf is a predecessor of} $b$.  
\end{definition}

Therefore, any partial order defines canonically its associated 
successor relation. This relation may be empty, as illustrated by  
$\Q$ or $\R$ with their usual orders. But on {\em finite} sets, it is easy to see that the knowledge 
of the successor relation is enough to reconstruct the initial partial order  
(see \cite[Section 2.4, Theorem 3]{BB 70}). 

Given a partial order on a finite set, it is more economic to encode the associated 
successor relation, 
as it has less pairs of related elements. A visual way to do this encoding is through its 
associated {\em geometric graph}, which is the geometric realization of the 
covering binary relation, seen as a {\em directed graph}. 
Let us first recall this last notion. 

   An abstract  {\bf directed graph} is a triple $(S,E, A)$ where 
   $S$ is a set of {\bf vertices}, $E$ is a set of {\bf edges} and 
   $A : E \rightarrow S \times S$ is a map. We denote  
    $A(e) = (s(e), t(e))$ and we say that the vertex 
   $s(e)$ is the {\bf source}, the vertex $t(e)$ is the 
   {\bf target} of the edge $e$ and that the two vertices are {\bf connected by} $e$. 
   If the map $A$ is injective (that is, 
   any pair of vertices is connected by at most one edge) and its 
   image is disjoint from the diagonal (that is, there are no {\bf loops}, which are edges 
   connecting a vertex with itself), 
   we say that we have a {\bf simple directed graph}.  
   A {\bf directed path} is a finite sequence of edges such that 
   the target of each edge is equal to the source of its successor. 
   A {\bf circuit} of a directed 
   graph is a directed path starting and ending at the same vertex.

To any simple directed graph $(S, E, A)$, 
one associates canonically 
its {\bf geometric realization} (a {\bf geometric simple directed graph}), which is a 
simplicial complex with $S$ and $E$ as sets of $0$-simplices and $1$-simplices 
respectively, each $e \in E$ corresponding to the oriented $1$-simplex 
$\overrightarrow{s(e)t(e)}$.

   For example, consider a binary relation $R$ defined on a set $S$, such that no 
   element of $S$ is related to itself.   
    The {\bf geometric graph of $R$}  
    is the geometric realization of the simple directed graph $(S, E, A)$, where 
   $E \subset S \times S$ is the set of pairs $(a,b)$ such that $a\ R \ b$ and 
   $A$ is the inclusion map.

Let us come back to partial order relations:

\begin{definition}\label{Hassediag}
   The {\bf Hasse diagram} of a poset $(S, \preceq)$ is the geometric graph of 
   the associated successor relation $\prec_s$. 
\end{definition}

\begin{example}
    Consider the poset of subsets of $\{1, 2, 3 \}$, ordered by inclusion. Its Hasse 
     diagram is represented in Figure \ref{fig:3elem}. 
\end{example}

\begin{figure}[h!] 
\vspace*{6mm}
\labellist \small\hair 2pt 
\pinlabel{$\emptyset$} at 120 5

\pinlabel{$\{1\}$} at 28 130
\pinlabel{$\{2\}$} at 125 130
\pinlabel{$\{3\}$} at 240 130

\pinlabel{$\{1,2\}$} at -25 260
\pinlabel{$\{2,3\}$} at 255 260
\pinlabel{$\{1,3\}$} at 92 260

\pinlabel{$\{1,2,3\}$} at 76 380

\endlabellist 
\centering 
\includegraphics[scale=0.45]{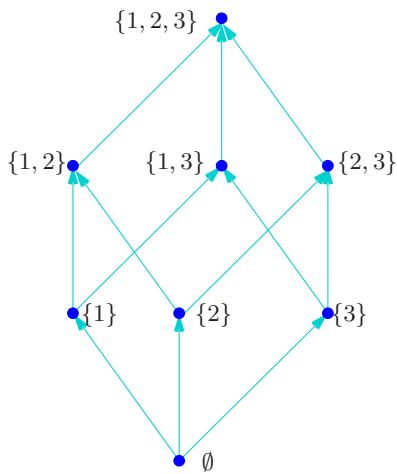} 
\caption{The Hasse diagram of the subsets of a set with three elements} 
\label{fig:3elem} 
\end{figure}

As stated in \cite[Section 2-10]{BB 70} and as it may be easily verified, 
one may characterize the simple directed graphs associated 
to partial orders on finite sets in the following way:

\begin{proposition} \label{charposet}
  A simple directed graph with {\em finite} vertex set $S$ 
 is the Hasse diagram of a partial order on $S$ if and only if it 
  has no circuits and  the only directed path joining the source and 
  the target of any edge  is the edge itself. 
\end{proposition}

Among posets, we will be particularly interested in {\em lattices}:

\begin{definition}  \label{latticedef}
   A {\bf lattice} is a poset $(S, \preceq)$ such that any two elements 
     $a, b \in S$ admit: 
  \begin{itemize} 
    \item a greatest lower bound denoted by $a \wedge b$ and called their 
           {\bf infimum}.
    \item a smallest upper bound denoted by $a \vee b$ and called their 
           {\bf supremum}.
  \end{itemize}
  A {\bf sublattice} of a given lattice is a subset which is closed under the 
  ambient  infimum and supremum operations. A lattice is {\bf distributive} 
  if each one of the two operations is distributive with respect to the other one.
\end{definition}

In the literature, $a\wedge b$ is also called the {\em meet} of $a$ and $b$, and 
$a \vee b$ is called their {\em join}. The knowledge of these two operations 
allows to reconstruct the partial order, as: 
   $$a \preceq b \ \Leftrightarrow \   a = a \wedge b \ \Leftrightarrow \ b = a \vee b. $$

\begin{example}
    The simplest examples 
of non-distributive lattices are the {\em diamond} and {\em pentagon} lattices, with Hasse diagrams 
drawn in Figure \ref{fig:Nondistr}. In both cases, 
$a \wedge (b \vee c) \neq (a \wedge b) \vee (a \wedge c)$. 
It is known that a lattice is distributive 
if and only if it does not contain any sublattice isomorphic to a diamond or a pentagon one. 
\end{example}

\begin{figure}[h!] 
\vspace*{6mm}
\labellist \small\hair 2pt 
\pinlabel{$o$} at 76 5
\pinlabel{$a$} at 18 96
\pinlabel{$b$} at 74 96
\pinlabel{$c$} at 130 96
\pinlabel{$i$} at 76 185

\pinlabel{$o$} at 327 6
\pinlabel{$a$} at 236 113
\pinlabel{$b$} at 236 60
\pinlabel{$c$} at 380 77
\pinlabel{$i$} at 310 185

\endlabellist 
\centering 
\includegraphics[scale=0.50]{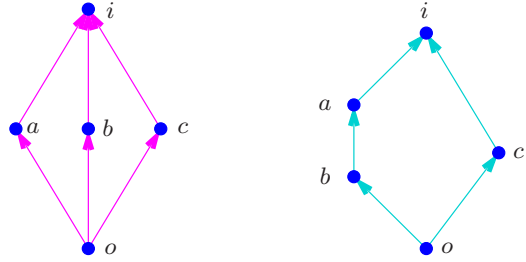} 
\caption{The Hasse diagrams of the diamond and the pentagon lattices} 
\label{fig:Nondistr} 
\end{figure} 

As other examples of posets, let us quote:
\begin{enumerate}
   \item the power set $\mathcal{P}(L)$ of a given set $L$, with the relation of inclusion; 
   \item the set $\N^*$ of positive integers, with the relation of divisibility; 
    \item the set of subgoups of a given group $G$, with the relation of inclusion; this is {\em not} 
      a sublattice of $\mathcal{P}(G)$, as the union of two subgroups is not in general a group.
\end{enumerate}

The first two cases are distributive lattices. In the first case, one has moreover 
a structure of Boolean algebra. 

Of course,  {\em any sublattice of a distributive lattice is also distributive}. In particular, 
any sublattice of a power set $(\calP(L), \subseteq)$ is distributive. Conversely, if a finite 
lattice $(S, \preceq)$ is distributive, then it may be embedded as a sublattice of a power set. 
Such an embedding may be obtained in a canonical way. In order to explain this, we need one more 
definition: 

\begin{definition} \label{def:infirr}
   Let $(S, \preceq)$ be a lattice. An element $i \in S$ is called {\bf sup-irreducible} 
   if it cannot be written as $a \vee b$, with $a$ and $b$ distinct from $i$. 
\end{definition}

The canonical realization of a distributive lattice as a sublattice of a power set is described by 
the following result (see \cite[Chapter IX.4]{B 48} or \cite[Chapter 7, Theorem 9]{G 71} 
for the proof and 
Definition \ref{posetdef} for the notion of hereditary subset of a poset): 

\begin{proposition}  \label{prop:canreal}
      Let $(S, \preceq)$ be a finite distributive lattice. Let $I^{\vee}$ be its subset 
       of sup-irreducible 
      elements. If $a$ is any element of $S$, denote by $\rho(a)$ 
      the subset of $I^{\vee}$ consisting of the sup-irreducible elements which are 
      less or equal to $a$. Then the map $\rho : S \to \calP(I^{\vee})$ embeds $(S, \preceq)$ as a 
      sublattice of $(\calP(I^{\vee}), \subseteq)$. Its image consists of the hereditary subsets of 
      the poset $(I^{\vee}, \preceq)$. 
\end{proposition}

We have now enough material to proceed to the proof of the uniqueness of the 
duality of $(\calE_n, \preceq)$.

\medskip
\section{The uniqueness of the duality of plane branches}
\label{sec:Uniq}

If a finite lattice admits an order-reversing bijection, 
  such a bijection is not necessarily unique. 
   For instance, as the reader may easily check, the diamond lattice 
   of Figure \ref{fig:Nondistr} has six such bijections. 
In this section we show that the duality $\Delta_n$ defined in Section \ref{sec:Latt} 
is the {\em unique} bijection of $\calE_n$ which reverses the staircase 
partial order structure. 

\medskip

\begin{figure}[h!] 
\vspace*{6mm}
\labellist \small\hair 2pt 
\pinlabel{$e_2$} at 95 130
\pinlabel{$v_2$} at 95 330

\endlabellist 
\centering 
\includegraphics[scale=0.30]{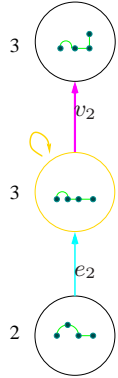} 
\caption{The Hasse diagram of the lattice structure on $\calE_4$} 
\label{fig:Fib-dual-c4} 
\end{figure}

\begin{figure}[h!] 
\vspace*{6mm}
\labellist \small\hair 2pt 
\pinlabel{$e_2$} at 510 128
\pinlabel{$e_3$} at 345 128

\pinlabel{$e_2$} at 360 326
\pinlabel{$e_3$} at 535 326
\pinlabel{$v_3$} at 205 326

\pinlabel{$e_2$} at 120 505
\pinlabel{$v_2$} at 480 505
\pinlabel{$v_3$} at 350 505

\pinlabel{$v_2$} at 185 680
\pinlabel{$v_3$} at 360 680

\endlabellist 
\centering 
\includegraphics[scale=0.35]{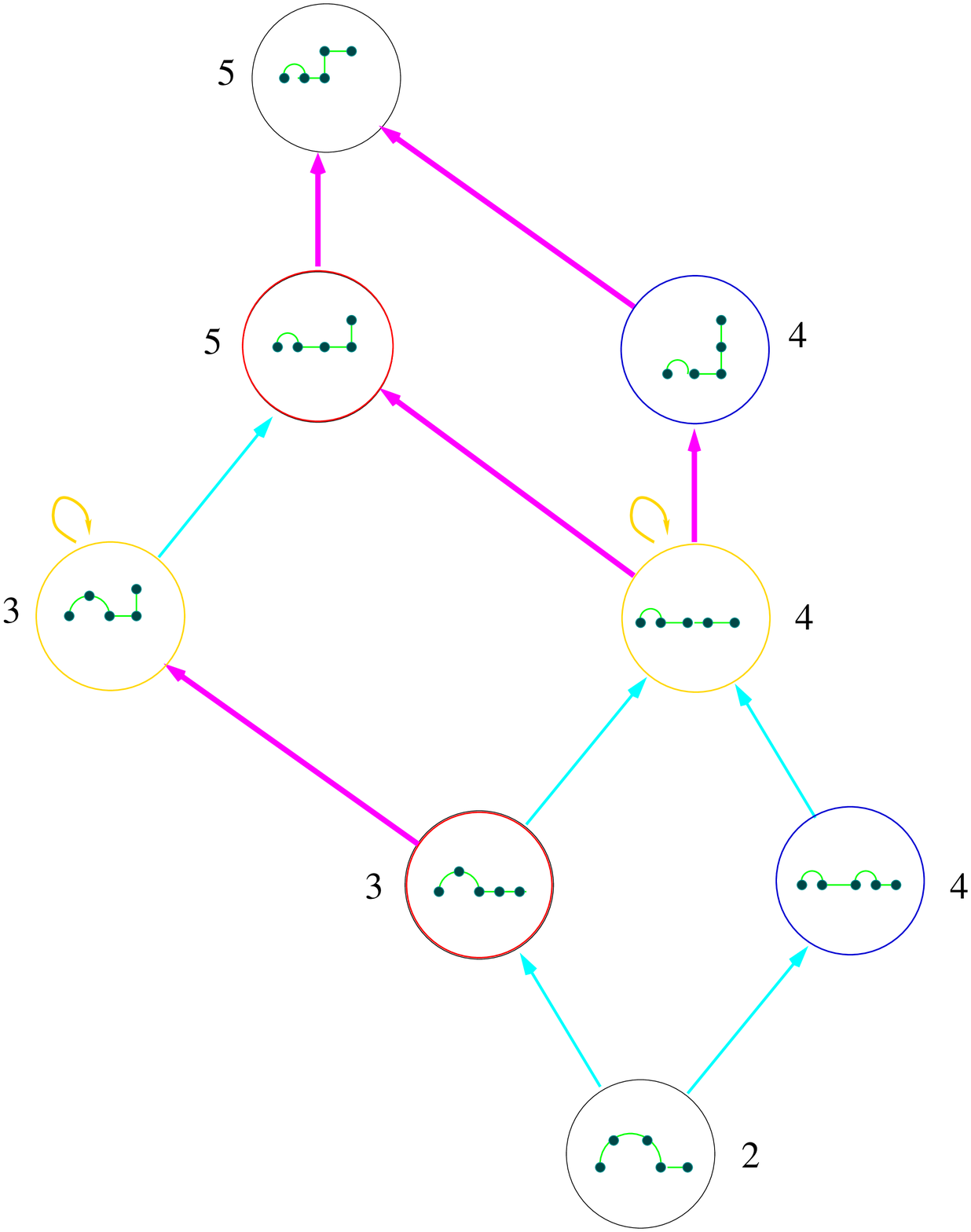} 
\caption{The Hasse diagram of the lattice structure on $\calE_5$} 
\label{fig:Fib-dual-c5} 
\end{figure}

\begin{figure}[h!] 
\vspace*{6mm}
\labellist \small\hair 2pt 
\pinlabel{$e_2$} at 528 150
\pinlabel{$e_3$} at 443 150
\pinlabel{$e_4$} at 333 150

\pinlabel{$v_4$} at 168 350
\pinlabel{$e_3$} at 293 350
\pinlabel{$e_4$} at 375 350
\pinlabel{$e_2$} at 440 350
\pinlabel{$e_2$} at 450 290
\pinlabel{$e_4$} at 513 290
\pinlabel{$e_3$} at 585 290

\pinlabel{$e_3$} at 45 475
\pinlabel{$e_2$} at 138 475
\pinlabel{$v_4$} at 245 475
\pinlabel{$v_3$} at 310 475
\pinlabel{$e_2$} at 425 520
\pinlabel{$v_4$} at 275 520
\pinlabel{$e_3$} at 480 475
\pinlabel{$e_4$} at 545 475
\pinlabel{$v_2$} at 660 475

\pinlabel{$v_3$} at 55 665
\pinlabel{$e_2$} at 145 665
\pinlabel{$e_3$} at 190 665
\pinlabel{$v_4$} at 270 665
\pinlabel{$e_2$} at 425 716
\pinlabel{$v_4$} at 300 716
\pinlabel{$v_3$} at 495 665
\pinlabel{$v_2$} at 560 665
\pinlabel{$e_4$} at 695 716

\pinlabel{$e_2$} at 120 830
\pinlabel{$v_3$} at 195 830
\pinlabel{$v_2$} at 295 830
\pinlabel{$v_4$} at 375 830
\pinlabel{$v_2$} at 580 880
\pinlabel{$v_4$} at 505 880
\pinlabel{$v_3$} at 680 880

\pinlabel{$v_2$} at 335 1070
\pinlabel{$v_3$} at 435 1070
\pinlabel{$v_4$} at 535 1070

\endlabellist 
\centering 
\includegraphics[scale=0.35]{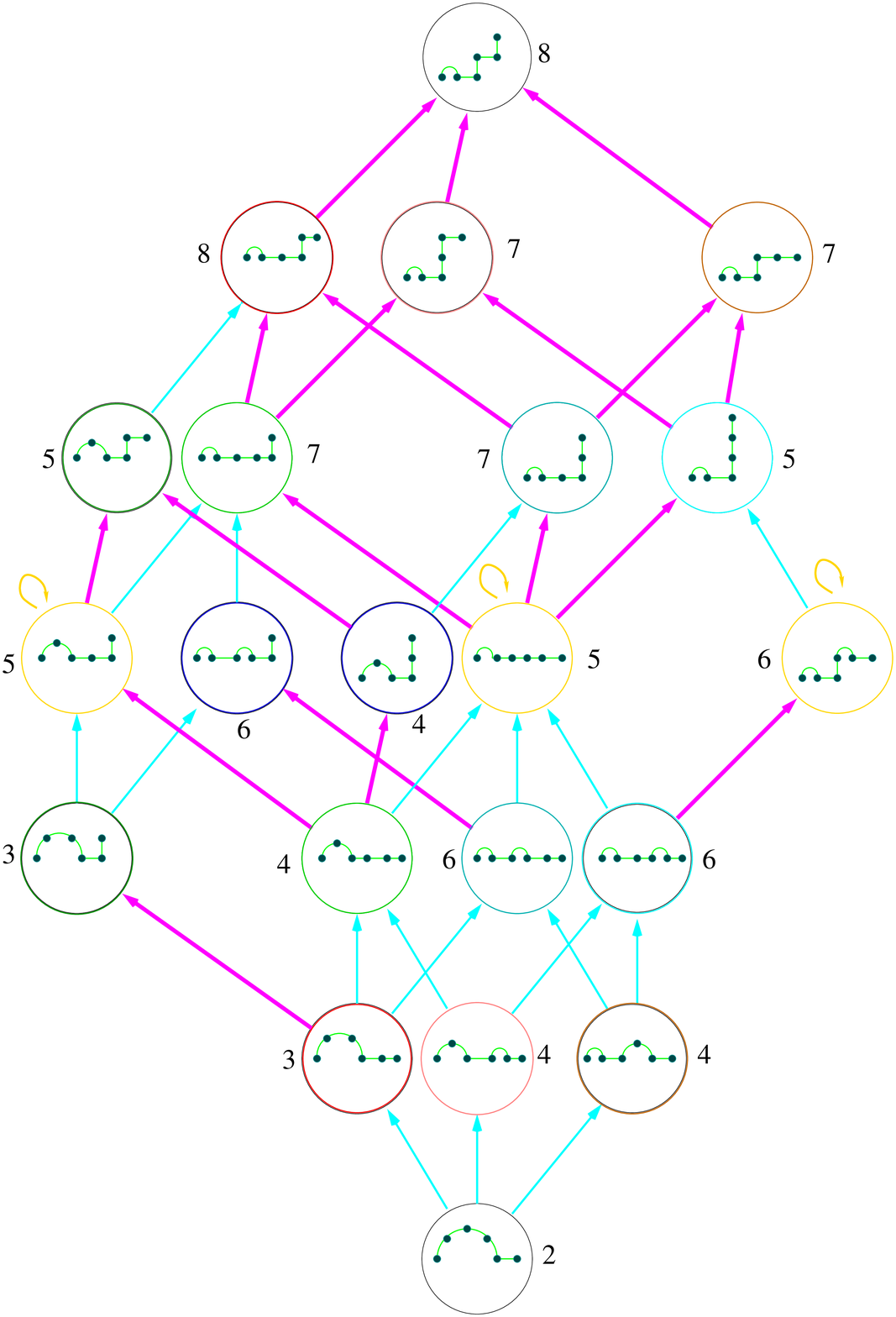} 
\caption{The Hasse diagram of the lattice structure on $\calE_6$} 
\label{fig:Fib-dual-c6} 
\end{figure}

Before proving the announced uniqueness, let us draw the Hasse diagrams 
of a few posets $(\calE_n, \preceq)$.

\begin{example}
  The case $n=3$ is trivial, as $\calE_3$ has only 
one element $\alpha_3 = \omega_3$.  
  In the Figures \ref{fig:Fib-dual-c4}, \ref{fig:Fib-dual-c5} and  \ref{fig:Fib-dual-c6} are 
represented the Hasse diagrams of the lattices $(\calE_4, \preceq)$, 
$(\calE_5, \preceq)$ and  $(\calE_6, \preceq)$ respectively. If an arrow 
 goes from a diagram $\epsilon$ to a greater diagram 
$\epsilon'$, it is labeled by the unique element of $\chi(\epsilon') \setminus 
\chi(\epsilon)$. The diagrams which have a self-referencing arrow are the 
self-dual ones. Near each diagram we indicate the corresponding initial multiplicity. 
\end{example}

As a first consequence of the constructions of Section \ref{sec:Latt} 
and of the definitions of Section \ref{sec:Basic}, one has: 

\begin{proposition} \label{distrlat}
    The staircase partial order relation of Definition \ref{def:Latstr} endows each set 
    $\calE_n$ with a structure of distributive lattice.
\end{proposition}

\begin{proof}
    By Lemma \ref{lem:Stable}, the set $\mathcal{K}_n$ of codes of Enriques diagrams 
    of complexity $n$ is a sublattice of 
    $(\calP(S_n), \subseteq)$. As this last lattice is distributive, we deduce  
    that the first one has also this property. By the Definition \ref{def:Latstr} 
    of the staircase partial order relation, 
    we conclude that $(\calE_n, \preceq)$ is indeed a distributive lattice.
\end{proof}

\begin{figure}[h!] 
\vspace*{6mm}
\labellist \small\hair 2pt 
\pinlabel{$e_2$} at 4 -20
\pinlabel{$e_3 $} at 150 -20
\pinlabel{$e_{n-3} $} at 365 -20
\pinlabel{$e_{n-2}$} at 509 -20

\pinlabel{$v_2$} at 76 100
\pinlabel{$v_3$} at 220 100
\pinlabel{$v_{n-3}$} at 436 100
\pinlabel{$v_{n-2}$} at 580 100

\endlabellist 
\centering 
\includegraphics[scale=0.60]{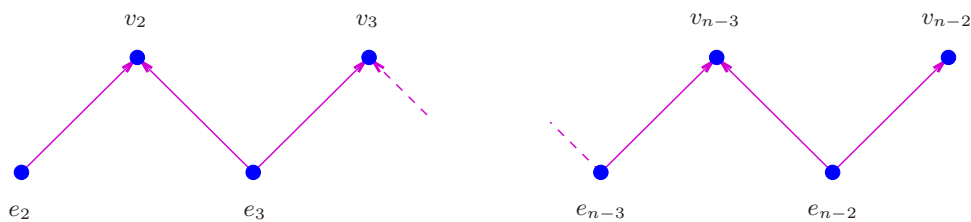} 
\vspace*{3mm} 
\caption{The Hasse diagram of the poset $(S_n, \preceq)$} 
\label{fig:Hasse-irred} 
\end{figure}

Note that 
   by Proposition \ref{prop:canreal}, the apparently ``complicated'' lattices $(\calE_n, \preceq)$ 
   are completely described up to lattice-isomorphisms by the much ``simpler'' 
   posets $(I_n^{\vee}, \preceq)$. 
   As results immediately from Lemma \ref {lem:Image},  the elements of $I_n^{\vee}$ are: 
   \begin{itemize} 
        \item $\{ e_i\}$ for $i \in \{2, \dotsc, n-2 \}$; 
        \item $\{ v_i, e_i, e_{i + 1} \}$ for $i \in \{2, \dotsc, n-3 \}$ and $\{ v_{n-2}, e_{n-2} \}$. 
   \end{itemize}
  
   One has the following relation between the code of an 
   Enriques diagram (see Definition \ref{code}) and its associated hereditary 
   subset of $(I_n^{\vee}, \preceq)$ (see Proposition \ref{prop:canreal}): 
   
   \begin{proposition}  \label{corresp}
       The bijection from  $S_n$ to 
   $I_n^{\vee}$ which sends each $e_i$ to $\{ e_i \}$,  
   each $v_i$ to  $\{ v_i, e_i, e_{i + 1} \}$ and $v_{n-2}$ to 
   $\{ v_{n-2}, e_{n-2} \}$, transforms the code $\chi(\epsilon) \in \calP(S_n)$ 
   of an Enriques  diagram $\epsilon \in \calE_n$ into the hereditary subset 
   $\rho(\epsilon) \in \calP(I_n^{\vee})$. Therefore, if one considers the 
   poset structure  $(S_n, \preceq)$ inherited from $(I_n^{\vee}, \preceq)$ 
   by the previous bijection, its Hasse diagram is as drawn in Figure 
   \ref{fig:Hasse-irred} and $\mathcal{K}_n$ is exactly the set of hereditary 
   subsets of $S_n$. 
    \end{proposition} 

\begin{proof}
    The statement about the Hasse diagram is checked easily using the definition 
    of the bijection. Then one checks using Lemma \ref{lem:Image} the statement 
    about the correspondence between codes and hereditary subsets of 
    $(I_n^{\vee}, \preceq)$. 
\end{proof}

\begin{figure}[h!] 
\vspace*{6mm}
\centering 
\includegraphics[scale=0.35]{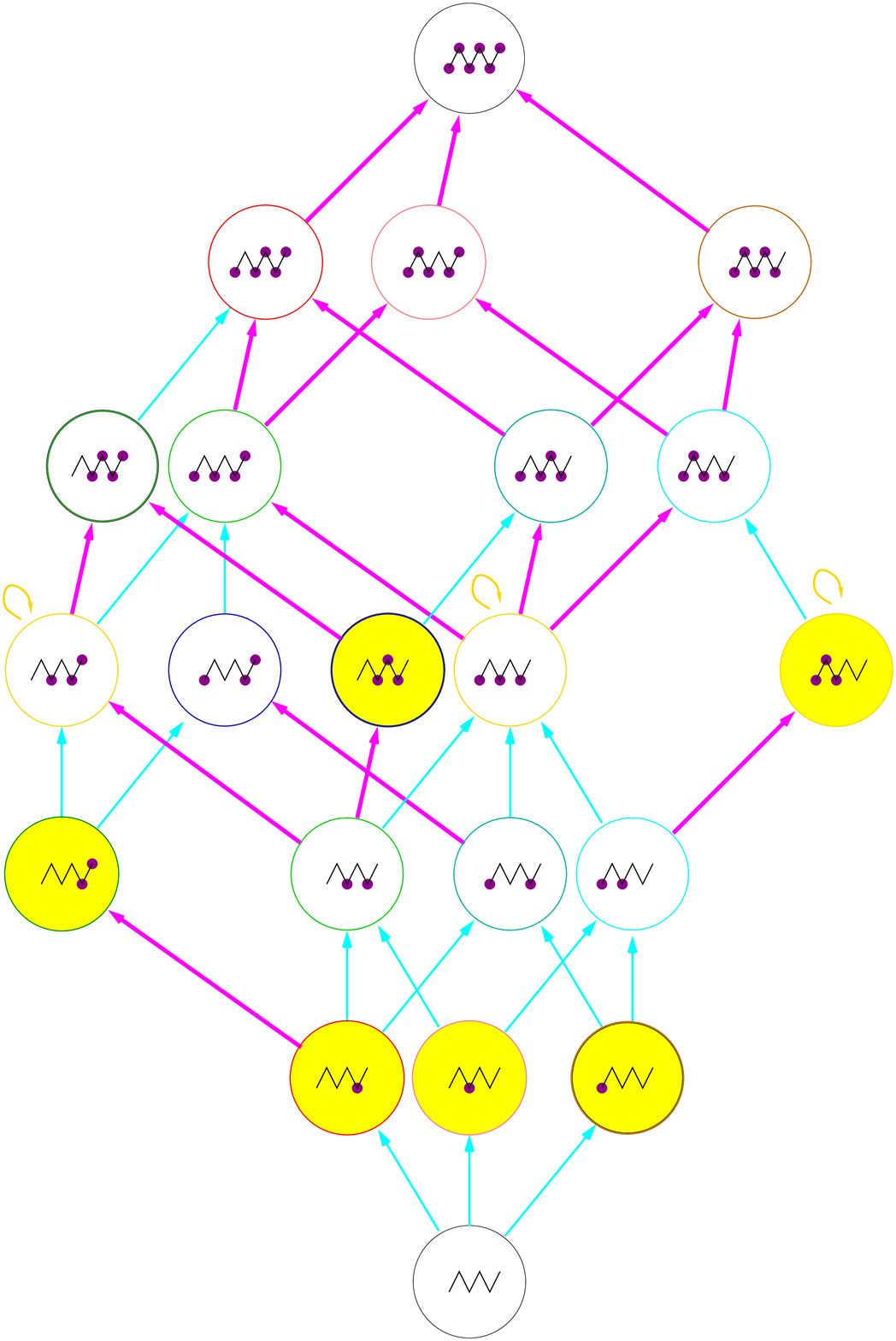} 
\vspace*{3mm} 
\caption{The lattice structure on $\calE_6$ encoded using its sup-irreducible elements} 
\label{fig:Fib-dual-irred} 
\end{figure} 

\begin{example}
  In Figure \ref{fig:Fib-dual-irred} is represented again the lattice of Figure 
   \ref{fig:Fib-dual-c6}. This time, we have represented each element of the lattice 
    $(\calE_6, \preceq)$ 
   through the associated hereditary subset of $(S_6, \preceq)$. 
   Each hereditary subset is represented 
   as a set of discs at the vertices of the Hasse diagram of $(S_6, \preceq)$.
   The sup-irreducible elements of $(\calE_6, \preceq)$, which correspond 
   therefore to the elements of $S_6$ by the bijection defined in Proposition 
   \ref{corresp}, are represented against a colored background.
\end{example}

The previous considerations allow us to prove the announced uniqueness theorem:

\begin{theorem} \label{uniqdual}
    The involution $\Delta_n$ is the unique bijection 
   of the set $\calE_n$ on itself which reverses the staircase partial order. 
\end{theorem}

\begin{proof}
    By our definition of the map $\Delta_n$, it is equivalent to show the analogous property 
    for the poset 
    $(\mathcal{K}_n, \subseteq)$ of codes of the Enriques diagrams of complexity $n$  
    and the corresponding involution, which we denote 
    again by $\Delta_n$ (see Definition \ref{def:dual}). 
    Assume that there is another order-reversing bijection $\Delta'_n$ of $\mathcal{K}_n$. 
    Then $\Delta_n \circ \Delta_n'$ is an automorphism 
     of this  poset. In particular, it restricts to an automorphism of the subposet 
     $(I_n^{\vee}, \preceq)$ of its sup-irreducible elements (see Definition \ref{def:infirr}). 
     
     By Proposition \ref{corresp}, this gives an automorphism of  
     the directed graph of Figure \ref{fig:Hasse-irred}. As $e_2$ is the only 
     vertex from which starts only one edge, it is fixed by this automorphism. 
     Looking then at the distance to this vertex in the Hasse diagram, one sees that 
     all vertices are fixed. 
     
     Therefore, the automorphism  
     $\Delta_n \circ \Delta_n'$ is the identity when restricted to 
     $(I_n^{\vee}, \preceq)$. Proposition \ref{prop:canreal} implies  
      that $\Delta_n \circ \Delta_n'$ is also the identity on $\mathcal{K}_n$, which shows that 
      $\Delta_n = \Delta_n'$. 
   \end{proof}

\begin{remark} \label{rem:rinrem}
     We discovered the results of this paper by thinking about the problem of 
     {\em adjacency of singularities}. A basic remark is that if the combinatorial 
     type $\epsilon_2$ of a plane branch appears on the generic fibers of a one parameter 
     deformation of  
     a plane branch of combinatorial type $\epsilon_1$ (one says then that they are {\em adjacent}), 
      then $m_0(\epsilon_2) \leq 
       m_0(\epsilon_1)$. That is why we decided to understand the behaviour of the 
     initial multiplicity $m_0$ on the set of combinatorial types of plane branches. 
     We started our study by restricting to combinatorial types of fixed blow-up 
     complexity. 
   We checked whether the initial multiplicity was increasing for a slightly 
   different definition of 
  straightening operator than the one of Definition \ref{def:operators}, in which the 
  neighboring edges are treated symmetrically. That is, if any one of them is straight, then the new 
 straight edge is aligned with it. If one draws then an arrow from each Enriques diagram 
 to every diagram obtained from it by one of the two types of operators, one obtains precisely 
 the Hasse diagrams of the staircase partial orders! On that of $\calE_6$ 
   (see Figure \ref{fig:Fib-dual-c6}), the duality 
 jumped to our eyes, which led us to prove that it was a general phenomenon. 
   One may also see on Figure \ref{fig:Fib-dual-c6} that the initial multiplicity 
  is not necessarily increasing for the staircase partial order on $\calE_6$. 
  In fact, and we leave this as an exercise for the reader, it is never increasing 
   on $(\calE_n, \preceq)$, for $n \geq 6$. Nevertheless, we hope that the lattice structure,  
   the duality and the straightening and breaking operators will help describing the pairs 
   $(\epsilon_1, \epsilon_2)$ of combinatorial types of adjacent singularities. 
\end{remark}

\medskip
\end{document}